\documentclass[11pt]{article}
\usepackage{amsfonts}
\usepackage{amsmath,amsthm,amscd,amssymb,mathrsfs,setspace}
\usepackage{latexsym,epsf,epsfig}
\usepackage{color}
\usepackage[hmargin=2.5cm,vmargin=3cm]{geometry}
\usepackage{hyperref}
\usepackage{cite}
\usepackage[dvipsnames]{xcolor}
\newcommand{\ds}{\displaystyle}

\newcommand{\realstwo}{\mathbb{R}^2}
\newcommand{\realsthree}{\mathbb{R}^3}
\newcommand{\xb}{{\bf{x}}}

\newcommand{\Dx}{{\partial_x}}

\newcommand{\cA}{{\mathscr{A}}}

\newcommand{\Dn}{\partial_{\nu}}

\newcommand{\cE}{{\mathcal{E}}}
\newcommand{\cD}{\mathscr{D}}

\newcommand{\R}{\mathbb{R}}

\newcommand{\bH}{\mathbf{H}}

\newcommand{\Ez}{E_{z}}

\theoremstyle{plain}
\newtheorem{theorem}{Theorem}[section]
\newtheorem{lemma}[theorem]{Lemma}
\newtheorem{proposition}[theorem]{Proposition}
\newtheorem{condition}{Condition}
\newtheorem{corollary}[theorem]{Corollary}

\newtheorem{definition}{Definition}

\theoremstyle{remark}
\newtheorem{remark}{Remark}[section]
\numberwithin{equation}{section}
\numberwithin{theorem}{section}
\numberwithin{remark}{section}
\numberwithin{assumption}{section}
\numberwithin{condition}{section}

\begin{document}

\title{Large Deflections of A Structurally Damped Panel \\ in A Subsonic Flow}
 \author{\begin{tabular}[t]{c@{\extracolsep{1em}}c}
         Abhishek Balakrishna & Justin T. Webster  \\ 
 \it University of Maryland, Baltimore County ~~&~~ \it University of Maryland, Baltimore County \\ 
 \it Baltimore, MD &\it Baltimore, MD \\  
bala2@umbc.edu & websterj@umbc.edu
\end{tabular}}

\maketitle

\begin{abstract}
\noindent  The large deflections of panels in subsonic flow are considered. Specifically, a fully clamped von Karman plate accounting for both rotational inertia in plate filaments and structural damping of square root type is considered. The panel is taken to be embedded in the boundary of a linear, subsonic potential flow on the positive halfspace in $\mathbb R^3$. Solutions are constructed via a semigroup approach despite the lack of natural dissipativity associated to the generator of the linear dynamics. The flow-plate dynamics are then reduced---via an explicit Neumann-to-Dirichlet (downwash-to-pressure) solver for the flow---to a memory-type dynamical system for the plate. For the non-conservative plate dynamics, a global attractor is explicitly constructed via Lyapunov and quasi-stability methods. Finally, it is shown that via the compactness of the attractor and finiteness of the dissipation integral, that all trajectories converge strongly to the set of stationary states. 
\vskip.15cm

\noindent {\em Key Terms}: mathematical aeroelasticity, von Karman plate, flutter, semigroup, quasi-stability, attractors, stabilization
\vskip.15cm
\noindent {\em 2010 AMS}: 74F10, 74K20, 76G25, 35B40, 35G25, 37L15 \end{abstract}

\section{Introduction}
In this treatment we present and discuss rigorous results for a panel flutter model appearing in the classic aeroelasticity literature \cite{dowellnon,B}, namely, a subsonic, inviscid potential flow interacting with the large deflections of a fully clamped plate. We specifically take a von Karman-type plate, allowing for rotational inertia effects in plate filaments as well as some (small) amount of structural damping. We are concerned with the Hadamard well-posedness of solutions (existence, uniqueness, and continuous dependence upon data), as well as qualitative properties of solutions beyond the transient regime. Specifically, we are interested in  aeroelastic instabilities such as {\em flutter}. 

Aeroelastic flutter, in a flow-plate system, is a particular type of feedback instability where the flow's aerodynamical loading at the fluid-structure interface destabilizes the otherwise stable damped plate \cite{B, book}. The flutter instability occurs as a bifurcation in the flow parameters, typically that of the unperturbed flow velocity $U$. Such an instability can manifest itself via chaotic plate oscillations \cite{dowellchaos,dowellrecent}, but often occurs in the form of {\em limit cycle oscillations}. Stationary instability---elastic bucking---is also possible, depending on the mechanical forcing and flow quantities in the system. Hence, a variety of interesting questions present themselves: (i) can one detect, from the parameters in the problem, if flutter will occur? (ii) is flutter, so to speak, suppressed by the inclusion of some form of mechanical damping? (iii) in what ways does the long-time behavior of the flow-plate system depend on the initial configuration? Engineers, for instance, often observe/remark {\em that panel flutter only occurs supersonically} \cite{B,dowellrecent}. In this paper, we attempt to address these points in a mathematically rigorous way, in the fully infinite dimensional setting, taken from the strongly coupled flow-plate PDE system presented below. 

Here, we will present the construction of both strong and weak PDE solutions through the semigroup approach, in particular characterizing the domain of the generator for the flow-plate dynamics. We utilize the structure of the potential flow equation to explicitly construct a particular Neumann-to-Dirichlet map for the flow, which permits a closed representation for the plate with memory effects scaled by the characteristic flow velocity. For the resulting non-conservative plate with memory, we construct a smooth, compact global attractor of finite dimension, so long as {\em some} structural damping is present in the plate. This attractor is global in the sense that it attracts bounded sets in the state space with uniform rate, i.e., does not depend on specific initial data. With the attractor in hand, we utilize the finiteness of the dissipation integral, along with the compactness of the plate-to-flow (Neumann) mapping, to show that subsonic trajectories always converge to the stationary set; this is to say, {\em we confirm the engineering assertion that subsonic panel systems have stationary end behavior}.

Many of the results presented in this paper come from a variety of places, e.g., \cite{springer,b-c,b-c-1,LBC96, book, websterlasiecka}. These references span books, book chapters, surveys, older papers, and newer papers, in some cases with sparse details. To our knowledge, a full rigorous discussion of this subsonic flow-panel system---from PDE model to well-posedness to attractors to stabilization---has not appeared. Thus, we choose a particular situation where a full exposition is possible.
Indeed, we consider the case of {a clamped von Karman plate} with rotational inertia effects included, as well as appropriately scaled (square root type) structural damping, precisely because the model permits a clean sequence of rigorous results and a linear discussion. The style of the paper is to provide all of the main results (and the lemmata on which they are built) in a formal mathematical way, without necessarily proving each of these supporting facts. In the relevant cases below, we clearly provide the reference for the proofs and further discussion. In some cases below, we provide new approaches and/or proofs for the system (e.g., making use of the more recent quasi-stability approach) which have not been applied to this model in the literature. And, in the case of our main stabilization result, we provide a detailed proof. 

\subsection{Panel Flutter Model}
The large deflections of an aeroelastic panel are typically modeled via the plate theory of von Karman \cite{ciarlet,springer}, going back to the early aeroelasticity literature \cite{bolotin,dowellnon}. Here, we also choose this cubic-type model based upon the quadratic strain-displacement law \cite{lagnese,ciarlet}. At equilibrium, we model the center line of the plate by a bounded domain\footnote{If $\Omega$ is a rectangle, no results below are affected.} $\Omega \subset \{x_3 = 0\}$ with smooth boundary $\Gamma$ having unit outward normal $\nu$. 

The inviscid potential flow corresponds to the linearization of compressible Navier-Stokes about the stationary state $U\mathbf e_1$, i.e., constant flow of velocity $U$ in the $x$-direction; we normalize the flow parameters so that $U=1$ corresponds to the speed of sound, i.e., Mach 1.  The flow environment we consider as $\realsthree_+ = \{\xb \in\realsthree\, :\, x_3> 0\}$ so that the plate's centerline $\Omega \subset \partial \mathbb R_+^3$.

In this situation, then, $u: \Omega \times [0,\infty) \to \mathbb R$ corresponds to the transverse plate deflections; $\phi: \mathbb R^3_+ \times [0,\infty) \to \mathbb R$ is the perturbation velocity potential, such that $\mathbf v = U\mathbf e_1+\nabla \phi$ is the perturbed flow field. Then, the evolution flow-plate system of interest here is given by:
\begin{equation}\label{flowplate}\begin{cases}
(1-\alpha\Delta)u_{tt}+\Delta^2u+k(1-\alpha\Delta)u_t(t) +f_v(u)= p_0+\big[\partial_t+U\partial_{x_1}\big]\phi\big|_{\Omega} & \text { in }~~ \Omega\times (0,T),\\
u(0)=u_0;~~u_t(0)=u_1 & \text{ in } ~~\Omega,\\
u=\Dn u = 0 & \text{ on } ~\partial\Omega\times (0,T),\\
(\partial_t+U\partial_{x_1})^2\phi=\Delta \phi & \text { in }~~ \realsthree_+ \times (0,T),\\
\phi(0)=\phi_0;~~\phi_t(0)=\phi_1 & \text { in }~~ \realsthree_+\\
\partial_{x_3} \phi = \big[(\partial_t+U\partial_{x_1})u\big]_{\text{ext}}& \text{ on } ~\{x_3=0\} \times (0,T).
\end{cases}
\end{equation}

Above, parameters such as mass, density, thickness, and stiffness have been scaled out. The remaining parameters are those relevant to this mathematical analysis: $U, \alpha, k$. Here, $\alpha>0$ corresponds to the accommodation of rotational inertia of plate filaments \cite{lagnese}, whereas $k >0$ corresponds to the presence of structural damping (of so called square-root type).  The function $p_0(\mathbf x)$ corresponds to a stationary pressure on the top surface of the plate. The notation $[\cdot]_{\text{ext}}$ above means extension by zero from $\Omega \to \mathbb R^2$ with corresponding restriction $r_{\Omega}[\cdot]$, and the standard trace operator denoted by $\gamma[\cdot]$ onto $\partial \Omega$ or $\{x_3=0\}$ is utilized (of course in the appropriate functional senses). 

The (scalar) von Karman nonlinearity \cite{lagnese,ciarlet,springer} is given through the von Karman bracket and the Airy stress function. The bracket is $$[u,w]=(\partial_{x_1}^2u)(\partial_{x_2}^2w)+(\partial_{x_2}^2u)(\partial_{x_1}^2w)-2(\partial_{x_1}\partial_{x_2}u)(\partial_{x_1}\partial_{x_2}w),$$ 
while the Airy function is defined as an elliptic solver, namely, $v=v(u)$ is the solution to
\begin{equation}
\Delta^2 v = -[u,u]~~  \text{in}~~\Omega~~,~~
v=\partial_{\nu}v=0 ~~ \text{on}~~\Gamma.
\end{equation}
Finally, letting $F_0(\mathbf x)$ represent a stationary planar force on $\Omega$ corresponding to in-plane plate loading (pre-stressing), we have {\em the von Karman nonlinearity}
\begin{equation} f_v(u)=-[u,v(u)+F_0].\end{equation} 
\begin{remark}
The damping above in \eqref{flowplate} we delineate as ``square-root" type, as it is in some sense an interpolation between weak damping of the form ~$+k_0u-t$, and Kelvin-Voigt damping of the form~ $+k_2\Delta^2u_t$. This type of damping it is popular in engineering since it most accurately reproduces physical (low) modal damping decay rates. See \cite[Section 2.1]{HHWW} for detailed discussion and further references.
\end{remark}

\subsection{Notation and Conventions} In this paper we utilize the standard notation and conventions for $L^p(\mathscr O)$ spaces and Sobolev spaces of order $s \in \mathbb R$, $H^s(\mathscr O)$ where $\mathscr O$ is some domain. The space $H_0^s(\Omega)$ denotes the completion of the test functions $C_0^{\infty}(\Omega)$ in the $H^s(\Omega)$ norm with dual $H^{-s}(\Omega)$. For our norm notation, we will denote $||\cdot||_{H^s(\mathscr O)} = ||\cdot||_s$, where the spatial domain will be clear from context; we will identify $||\cdot||_{L^2(\mathscr O)}=||\cdot||$, omitting $s=0$. Inner products on $\mathbb R^3_+$ will be denoted by $(\cdot,\cdot) := (\cdot,\cdot)_{L^2(\mathbb R_+^3)}$ and on $\partial \mathbb R^3_+$ we utilize the notation $\langle \cdot,\cdot\rangle := (\cdot,\cdot)_{L^2(\Omega)}.$ The trace operator on $H^1(\mathscr O)$ spaces will be denoted by $\gamma[\cdot]$ with range in $H^{1/2}(\partial \mathscr O)$. We denote an open ball of radius $R$ in a Banach space $X$ by $B_R(X)$.

\textbf{Throughout the entirety of this paper, unless otherwise explicitly stated, we consider $U \in [0,1)$.}

\subsection{Energies and Solutions}\label{energiessols}
The energetic constraints for solutions manifest themselves through natural topological requirements, namely, for $L_{\alpha}^2(\Omega)$ given by $$||\cdot||_{L_{\alpha}^2(\Omega)}^2 := \alpha ||\nabla \cdot||_{L^2(\Omega)}^2+||\cdot||_{L^2(\Omega)}^2,$$ {\em finite energy solutions} should have the properties: \begin{equation}\label{finenreq} u \in C(0,T; H_0^2(\Omega))\cap C^1(0,T;L_{\alpha}^2(\Omega)); ~~~~ \phi \in C(0,T;W_1(\realsthree_+))\cap C^1(0,T;L^2(\realsthree_+)),\end{equation}
where $W_1(\R^3_+)$ denotes the homogeneous Sobolev space of order $1$. Here, $$W_1(\realsthree_{+})=\left\{\phi \in L_{loc}^2(\realsthree_+) : \nabla\phi\in L^2(\realsthree_{+}) \right\},$$ which is to say the space topologized by the {\em gradient norm} $||\nabla \phi||_{L^2(\mathbb R_+^3)}$ without $L^2(\mathbb R^3_+)$ norm control.

   To set provide a dynamical systems framework, the principal state space is taken to be \begin{equation}\label{yspace}Y = Y_{fl}\times Y_{pl} \equiv \big(W_1(\realsthree_+) \times L^2(\realsthree_+)\big)\times\big(H_0^2(\Omega) \times L_{\alpha}^2(\Omega)\big),\end{equation}
 We will also consider a stronger space on finite time intervals: \begin{equation}\label{ysspace}Y_s \equiv H^1(\realsthree_+)\times L^2(\realsthree_+) \times H_0^2(\Omega) \times L_{\alpha}^2(\Omega).\end{equation}

The energies corresponding to finite energy solutions of \eqref{flowplate}, and the above space $Y$, are given below.
\begin{align}\label{energies}
E_{pl} = & ~\dfrac{1}{2}\big[||u_t||_{L_{\alpha}^2 (\Omega)}^2+ ||\Delta u ||^2+\frac{1}{2}||\Delta v(u)||^2\big]-\langle F_0,[u,u]\rangle+\langle p_0,u\rangle, \\
E_{fl}=  &~\dfrac{1}{2}\big[||\phi_t||^2+||\nabla \phi||^2-U^2||\partial_{x_1}\phi||^2\big],~~\\E_{int} =&  ~2U\langle\gamma[\phi],\partial_{x_1}u\rangle,\\
\mathcal E =&~ E_{pl}+E_{fl}+E_{int}.
\end{align}

The pair $(\phi,u)$ as in \eqref{finenreq} is said to be a {\em strong solution} to \eqref{flowplate} on $[0,T]$ if:
\begin{itemize}
	\item $(\phi_t,u_t) \in L^1\Big(a,b; H^1(\realsthree_+)\times H^2(\Omega)\cap H_0^1(\Omega)\Big)$ for any $(a,b) \subset [0,T]$.
	\item $(\phi_{tt},u_{tt}) \in L^1\Big(a,b; L^2(\realsthree_+)\times H_0^1(\Omega)\Big)$ for any $(a,b) \subset [0,T]$.
	\item $\phi(t)\in H^2(\realsthree_+)$ and $\Delta^2 u(t)+k(1-\Delta)u_t(t) \in H^{-1}(\Omega)$ a.e. $t\in [0,T]$.
	\item The equation ~~$(1-\alpha\Delta)u_{tt}+\Delta^2 u +k(1-\alpha\Delta)u_{t}+f_v(u)=p_0+r_{\Omega}\gamma[\phi_t+U\phi_x]$ holds in $H^{-1}(\Omega)$ a.e.  $t >0$.
	\item The equation ~~$(\partial_t + U\partial_x)^2\phi=\Delta \phi$ holds a.e. $t>0$ and a.e. $\xb \in \realsthree_+$.
	\item The boundary conditions in \eqref{flowplate} hold a.e. $t\in [0,T]$ and a.e. $\xb \in \Gamma$, $\xb \in \realstwo$ respectively. 
	\item The initial conditions are satisfied point-wisedly; that is ~$\phi(0)=\phi_0, ~\phi_t(0)=\phi_1, ~u(0)=u_0, ~u_t(0)=u_1.$
\end{itemize} Strong solutions are {\em point-wise} or classical solutions.

The pair $(\phi,u)$ is said to be a {\em generalized solution} to  problem (\ref{flowplate}) on the interval $[0,T]$ if there exists a sequence of strong solutions $(\phi^n(t);u^n(t))$ with some initial data $(\phi^n_0,\phi^n_1; u^n_0; u^n_1)$ such that
$(\phi^n,u^n)$ converge to $(\phi,u)$ in the sense of $C([0,T] ; Y_s)$ as $n\to \infty$. Such solutions correspond to semigroup solutions for initial data in $Y$ rather than the domain of the generator.

Lastly,  the pair $(u,\phi)$, with $$u \in \mathscr{W}_T\equiv \Big\{ u \in L^{\infty}\big(0,T;H_0^2(\Omega)\big),~\partial_t u(\xb,t)\in L^{\infty}\big(0,T;H_0^1(\Omega)\big)\Big\}$$$$\phi  \in \mathscr{V}_T\equiv \Big\{ \phi \in L^{\infty}\big(0,T;H^1(\realsthree_+)\big),~\partial_t \phi(\xb,t)\in L^{\infty}\big(0,T;L^2(\realsthree_+)\big)\Big\},$$
 is said to be a {\em weak solution} to \eqref{flowplate} on $[0,T]$ if
\begin{itemize}
	\item $u(\xb,0)=u_0(\xb),u_t(\xb,0)=u_1(\xb)$ 
	and $\phi(\xb,0)=\phi_0(\xb),\phi_t(\xb,0)=\phi_1(\xb)$
	\item{
	$\ds 
	\int_0^T\Big((1-\alpha\Delta)\langle\partial_t u(t),\partial_t w(t)\big\rangle -k(1-\alpha\Delta)\langle\partial_t u(t), w(t)\big\rangle -\big\langle\Delta u(t),\Delta w(t)\big\rangle 
	- \big\langle f_v(u(t))-p_0,w(t)\big\rangle  \\ -\big\langle r_{\Omega}\gamma[\phi(t)],\partial_t w(t)+U\partial_{x_1} w(t)\big\rangle \Big)~dt
	=\big\langle u_1-r_{\Omega}\gamma[\phi_0],w(0)\big\rangle_{L^2(\Omega)}$}
	\newline for all test functions $w\in \mathscr{W}_T$ with $w(T)=0$. 
	\item $\ds \int_0^T \Big[\big((\partial_t+U\partial_{x_1})\phi(t),(\partial_t+U\partial_{x_1})\psi(t)\big) -\big( \nabla \phi(t),\nabla \psi(t)\big) \\ \phantom{\hskip4cm}+\big\langle (\partial_t+U\partial_{x_1})u(t),r_{\Omega}\gamma[\psi(t)]\big\rangle\Big]dt
	=\big(\phi_1+U\partial_{x_1} \phi_0,\psi(0)\big) $ \newline for all test functions $\psi\in\mathscr{V}_T$ such that $\psi(T)=0$. 
\end{itemize}
It is clear that strong solutions are generalized, and we state without proof that generalized solutions are in fact weak---see the discussion of abstract second order equations in \cite{springer}.

\subsection{Outline and Overview of Results Presented Here}

In {\bf Section 2 } we rewrite the problem abstractly, as dictated by the principal spatial operators for the plate and flow equations. We show, via a semigroup approach, that the underlying linear problem is well-posed via the Lumer-Phillips Theorem. From there, the locally Lipschitz nature of the von Karman nonlinearity yields local-in-time strong and generalized solutions, which are made global by specific bounds on trajectories. We utilize tight control of lower order terms via the superlinear nature of the von Karman nonlinearity, as well as the Hardy inequality to control interactive, non-dissipative flow-plate terms. The global in time bounds on trajectories provide finiteness of the dissipation integral, a critical piece of the stability analysis to follow.

{\bf Section 3} describes the stationary problem associated to the flow-plate dynamics \eqref{flowplate}. We quote results about the existence of stationary solutions, and remark on the mutliplicity of such solutions in general, concluding that the stationary set is generically finite owing to the Sard-Smale theorem. 

In {\bf Section 4} we look at the decoupled Neumann-type wave equation corresponding to the subsonic flow, driven by a given downwash. We decompose the effects of initial and given boundary data, and discuss stability and Huygen's principle in this context. In fact, as the hyperbolic equation is posed on the half space, we have an explicit solution representation (via transform methods) for $\phi$ in terms of the Neumann data. 
{\bf Section 5} uses this explicit solution form to compute the Dirichlet trace of the material derivative of the flow potential (the pressure). Since the flow data are in fact taken from the plate equation, this calculation allows us to consider a closed plate system with a memory-type term (as well as damping and non-conservative terms). In essence, this {\em reduces} the flow-plate dynamics to a memory-type plate dynamics.

For the memory-type dynamical system corresponding to plate solutions we construct a compact global attractor in {\bf Section 6}. For this non-conservative dynamical system, we must explicitly construct the absorbing ball, which is done via Lyapunov methods. We utilize the quasi-stability approach to obtain asymptotic compactness (yielding the existence of the compact global attractor), and quasi-stability on the absorbing ball provides finite dimensionality of the attractor in the state space as well as additional smoothness.

Finally, in {\bf Section 7}, we present the main result here on stabilization of the dynamics to the equilibria set. After showing that plate trajectories in fact converge to stationary points, we then {\em lift} this convergence to the flow. This result utilizes: compactness of the attractor, finiteness of the dissipation integral, compactness of the specific Neumann-to-Dirichlet  map provided through the explicit solution representation discussed above.

The main results in each of the above sections are stated precisely in their respective sections. Proofs are provided in most cases, but where they are not, precise references are given. 

\subsection{Discussion of Results Herein and Relationship to the Literature}

Let us discuss the previous mathematical work on this and closely related models. Early engineering references address panel flutter (in the comparable formulation to \eqref{flowplate}) as motivated by the paneling and external layers on aircraft and projectiles \cite{dowellnon,B}. We make special note of the work of Bolotin \cite{bolotin}, whose early work has the mathematical formulation of the flow-plate system here, as  well as good mathematical insight into a variety of qualitative features of the dynamics. Later, the  work of Chueshov et al. began to address flow-plate models in a modern PDE and dynamical systems sense \cite{b-c,LBC96,b-c-1}---indeed, Chueshov should be given credit as a driving force for the analysis of this and other models of mathematical aeroelasticity. More broadly, we mention other seminal works in mathematical aeroelasticity \cite{bal0,HM}, as well as the surveys \cite{survey1,survey2} and the book chapter \cite{book} which provide an overview of mathematical aeroelasticity, including some modeling discussions for  configurations other than that of a panel.

Specific to the models described in this treatment,  we point to early work on the delay dynamical system as it appears here can be found in \cite{oldchueshov1,Chu92b}. Later, the works \cite{LBC96,b-c,b-c-1} consider the system presented here as \eqref{flowplate}. Well-posedness is addressed through Galerkin constructions with good microlocal estimates \cite{miyatake1973} applied by decoupling the flow-plate system. Later, stabilization-type results appeared for the flow plate system when beneficial thermal effects are accounted for in the plate \cite{ryz,ryz2}. In the monograph \cite{springer}, many results appeared for attractors for the plate system, though not explicitly using the more recently developed quasi-stability theory. A proof of convergence to equilibrium for the model discussed here was outlined in both  \cite{chuey,springer}, without details. In the case without rotational inertia---namely $\alpha=0$---well-posedness was obtained for the first time in \cite{webster} using a semigroup approach, and later again with boundary dissipation in \cite{websterlasiecka}. Stabilization to equilibria was considered in the $\alpha=0$ case (with only weak damping) in the sequence \cite{stab1} and \cite{stab2}; these results are much more complicated, the nonlinearity is not principally that of von Karman, and the results are in some sense partial. 

Thus, for the treatment at hand, we choose a physically relevant scenario where many results can be presented in a clean and clear manner, utilizing the state of the art for modern dynamical systems theory. The results here, in some cases, are not sharp (with respect to parameters, for instance), and not all of which are novel. In fact, many of the results presented here have appeared in disjointed publications listed above over the past 30 years (with varying degrees of detail). Some of the results presented here which have appeared before are given here with novel proofs. 

The highlights of what is presented here are:
\begin{itemize} \item  a semigroup approach to well-posedness of strong and generalized solutions to the flow-plate system. This is a recent treatment for  these strongly coupled,  non-dissipative dynamics, and treats the system as a whole, exploiting cancellations through calculations performed on strong solutions rather than relying on technical trace results for decoupled dynamics;
\item  a direct construction of an absorbing ball for the plate dynamics, required because of the non-gradient structure of the reduced dynamics, making use of a non-standard Lyapunov approach on the reduced plate dynamical system with memory;
\item the use of modern quasi-stability theory \cite{quasi} on the absorbing ball, yielding, all at once, a compact global attractor for plate dynamics that is also smoother than the finite energy space as well as of finite fractal dimensional;
\item a complete proof of subsonic convergence to equilibrium, utilizing the compactness of the plate attractor, finiteness of the dissipation integral, and the compactness of the plate-to-flow lifting.
\end{itemize}

\section{Well-posedness and Boundedness of Solutions}
In this section we will construct strong and generalized (and hence weak) solutions via a semigroup approach. This approach was first utilized in the case $\alpha=0$ in \cite{webster} and some subsequent references \cite{supersonic,websterlasiecka,book} we based on it (also for $\alpha=0$). We provide here the abstract setup and semigroup generation result for the model at hand, \eqref{flowplate} when $\alpha>0$. We will re-write the linear problem abstractly on the finite energy space, using appropriate constituent operators. The generator, then, for the underlying linear flow-plate dynamics is $\omega$-dissipative \cite{pazy} and maximal in the appropriate sense, yielding generation. Then, exploiting the local-Lipschitz property of the von Karman nonlinearity, we will obtain local-in-time solutions for the nonlinear problem. Lastly, a priori-estimates on solutions (which exploit the superlinear nature of the nonlinearity) provide global solutions in $Y_s$ (as in \eqref{ysspace}) on any $[0,T]$. Further energy estimates ensure that the solution is uniformly bounded in time in the extended space $Y$ as in \eqref{yspace}.

\subsection{Operators and Abstract Restatement of the Linear Problem}
Let $A:\mathscr{D}(A)\subset L^2(\realsthree_+) \to L^2(\realsthree_+)$ be the positive (recall: $0 \le U<1$), self-adjoint operator  $$Af=-\Delta f+U\partial^2_{x_1} f+\mu f~, \hskip1cm \mathscr{D}(A)=\big\{f \in H^2(\realsthree_+): \partial_{x_3}f \big|_{x_3=0}=0\big\},$$
where $\mu>0$.\footnote{The perturbation $\mu>0$ is introduced to dispense with the zero eigenvalue; later that this will be taken as a bounded perturbation on $Y$ and removed to obtain the problem as originally stated.}  It is clear, then, that  $\mathscr{D}(A^{1/2}) = H^1(\realsthree_+)$ (in the sense of topological equivalence).
The corresponding Neumann map $N_0: L^2(\Omega) \to L^2(\realsthree_+)$ obtained through Green's formula \cite{redbook} is given by \begin{equation} \label{Neumannmap} \psi=N_0 w \iff (-\Delta+U\partial_{x_1}^2+\mu)\psi=0 \text{ in } \realsthree_+ 
~\text{ with    } ~\dfrac{\partial \psi}{\partial x_3}\Big|_{\Omega}=[w]_{ext}.\end{equation}  We have from \cite{springer,redbook} that ~ $\displaystyle A^{3/4-\epsilon}N_0: L^2(\mathbb R^2) \to L^2(\realsthree_+)$ is continuous. Moreover, we have the adjoint identification of the Dirichlet trace ~
$$\displaystyle N_0^*A f=\gamma[f], ~f \in H^1(\realsthree_{+}).$$
Introduce now the differential operator $\displaystyle D(\phi) \equiv 2U\phi_x$ defined on $\mathscr D(A^{1/2})$ 
and the biharmonic operator ~$\cA u = \Delta^2 u $, ~defined on $\mathscr D (\mathscr A) = (H^4\cap H_0^2) (\Omega)$. 
In this case, $\mathscr A$ is a positive, self-adjoint operator on $L^2(\Omega)$ with
$D(\cA^{1/2} ) = H_0^2(\Omega)$.
Lastly, we define the operator $M_{\alpha}= (1-\alpha \Delta)$ with domain $\mathscr D (M_{\alpha})= (H^2\cap H_0^1)(\Omega)$ and $\mathscr D(M_{\alpha}^{1/2})=H_0^1(\Omega)=L^2_{\alpha}(\Omega)$.

Consider the above, we have the abstract formulation of the homogeneous linear version of \eqref{flowplate}:
\begin{equation}\label{abstractsystem} \begin{cases}
\phi_{tt}+A\big(\phi+N_0(u_t+Uu_{x_1})\big)+\mu \phi +D(\phi_t)=0, ~\text{ in }~[\cD(A^{1/2})]'\\ 
M_{\alpha} u_{tt}+kM_{\alpha}u_t+\mathscr{A}u-N_0^*A[\phi_t+U\phi_{x_1}]=0 , ~\text{ in }~[\cD(\cA^{1/2})]'\\
\phi(0)=\phi_0, ~\phi_t(0)=\phi_1, ~u(0)=u_0, ~u_t(0)=u_1.
\end{cases}
\end{equation}
The natural state space then becomes $$Y_{\alpha} \equiv Y_1 \times
Y_{2,\alpha} =\mathscr{D}(A^{1/2}) \times L^2(\realsthree_+) \times \mathscr{D}(\mathscr{A}^{1/2}) \times 
\mathscr{D}(M_{\alpha}^{1/2}),$$ 
with natural inner product:
$y=(\phi_1, \phi_2; u_1, u_2)$ and $y'=(\phi_1', \phi_2'; u_1', u_2')$ 
\begin{align}\label{inner1} (y,y')_{Y_{\alpha}} =&(\phi_1,\phi_1')_{\mathscr{D}(A^{1/2})}+(\phi_2,\phi_2')_{L^2(\mathbb R^3_+)}
+( u_1,u_1')_{\mathscr{D}(\mathscr{A}^{1/2})}+(
u_2,u_2')_{\mathscr D(M_{\alpha}^{1/2})}.
\end{align}
For $\alpha>0$ fixed and any $\mu>0$, we have immediately that $Y_{\alpha}=Y_{s}$ in the sense of topological equivalence.

Writing (\ref{abstractsystem}) as a first order system for $y=(\phi_1,\phi_2,u_1,u_2)^T$ leads to
the overall dynamics operator $T_{\alpha}: \mathscr{D}(T_{\alpha})\subset Y_{\alpha} \to Y_{\alpha}$ expressed (in a distributional sense) by:
\begin{align} \label{T}
T_{\alpha} y=
\begin{bmatrix}
0 & -I & 0 &0\\ 
A & D& UAN_0\partial_{x_1} & AN_0\\
0&0&0&-I\\
- M^{-1}_{\alpha}UN_0^*A\partial_{x_1} & -M^{-1}_{\alpha}N_0^*A  &M^{-1}_{\alpha} \mathscr{A} & kI
\end{bmatrix}\begin{pmatrix}
\phi_1\\
\phi_2\\
u_1\\
u_2
\end{pmatrix},
\end{align}
with \begin{align} \mathscr{D}(T_{\alpha})  = \Big\{ y \in \big[\mathscr{D}(A^{1/2})\big]^2 \times \big[\mathscr{D}(\mathscr{A}^{1/2}]\big]^2:&~ \phi_1 + N_0 (u_2 + U \Dx u_1) \in \mathscr D(A); \\~ 
&~M^{-1/2}_{\alpha} [\mathscr{A}u_1-N_0^*A(\phi_2+U\partial_{x_1} \phi_1)]\in L^2(\Omega)\Big\}.\nonumber\end{align} 
In this case, the linearized, homogeneous version of \eqref{flowplate}, perturbed by $\mu$,  is represented by the abstract ODE \begin{equation}\label{firstorder}
\dfrac{dy}{dt}+T_{\alpha}y=0; ~ y(0)=y_0 \in Y_{\alpha}
\end{equation}

\subsection{Generation}
In this section we use the Lumer-Philips theorem \cite{pazy} to obtain:
\begin{theorem}
The operator $T_{\alpha}: \mathscr D(T_{\alpha}): Y_{\alpha} \to Y_{\alpha}$ is $\omega$-accretive and maximal, hence $-T_{\alpha}$ is the generator of a strongly continuous semigroup of bounded linear operators on $Y_{\alpha}$ (and hence also on $Y_{s}$). 
\end{theorem}
\begin{proof}[Proof Outline]
To show that $T_{\alpha}$ is $\omega$-accretive, we first consider a modified inner-product on $Y_{s}$:
For $y=(\phi_1, \phi_2; u_1, u_2)^T$ and $y'=(\phi_1', \phi_2'; u_1', u_2')^T$ 
\begin{align*} ((y,y'))\equiv &~(y,y')_{Y_{\alpha}}+U\langle \partial_{x_1} u_1,r_{\Omega} \gamma[\phi_1']\rangle+U\langle \partial_{x_1} u_1',
r_{\Omega}\gamma[\phi_1]|_{\Omega}\rangle+\lambda\langle \nabla u_1,\nabla u_1'\rangle, \end{align*} for  $\lambda = \lambda (U)$ a parameter chosen to ensure positivity of the inner-product $((\cdot,\cdot))$. It is a straightforward exercise using the Sobolev embedding theorems and the Hardy inequality to check that $((\cdot,\cdot))$ is in fact an inner product on $Y_{s}$ whose topology is equivalent to that given by the original inner product $(\cdot,\cdot)_{Y_{\alpha}}$. 

This inner-product $((\cdot,\cdot))$ on $Y_{\alpha}$ is built to produce a particular cancellation of trace terms in the accretivity calculation. Indeed, one my check that for $\displaystyle\omega>\frac{\lambda(U)}{2}$: $$\big(\big(T_{\alpha} y+\omega y ,y\big)\big)\ge 0,  ~
y \in \mathscr{D}(T_{\alpha}),$$ and hence $T_{\alpha}$ is $\omega$-dissipative in $((\cdot,\cdot))$ on $Y_{\alpha}$.

To show that $T_{\alpha}$ is maximal, we need to show that $\mathscr{R}(T_{\alpha}+\eta I)=Y_{s}$ for some $\eta>0$.
Given $x=(\psi_1,\psi_2; w_1, w_2)^T\in Y_{s}$ we must solve $\eta y+Ty=x$ for $y=(\phi_1,\phi_2;u_1,u_2)^T\in\mathscr{D}(T_{\alpha})$:
\begin{equation}\label{rangecondition}
\begin{cases}
\eta \phi_1-\phi_2&=\psi_1\in \mathscr{D}(A^{1/2})\\
\eta \phi_2+A[\phi_1+N_0(u_2+U\partial_{x_1} u_1)]+D(\phi_2)&=\psi_2\in L^2(\realsthree_+)\\
\eta u_1-u_2 &=w_1\in\mathscr{D}(\mathscr{A}^{1/2})\\
(\eta+k) u_2+ M^{-1}_{\alpha} [\mathscr{A}u_1-N_0^*A(\phi_2+U\partial_{x_1} \phi_1)]&=w_2\in
L^2_{\alpha}(\Omega).
\end{cases}
\end{equation} Eliminating $\phi_1$ and $u_1$, and applying $M_{\alpha}$ to
the last equation in (\ref{rangecondition}), we obtain the operator
\begin{equation}\label{L}
\mathscr G=
\begin{pmatrix}
\dfrac{1}{\eta} A+D+\eta I&AN_0(I+\dfrac{U}{\eta} \partial_x)\\[.25cm]
- N_0^*A(I+\dfrac{U}{\eta} \partial_{x_1})&\dfrac{1}{\eta}  \mathscr{A}+(\eta+k)  M_{\alpha} \end{pmatrix},
\end{equation} 
and the equation
\begin{equation}\label{rangeconditionreduced}
\mathscr G\begin{pmatrix}
\phi_2\\[.25cm]
u_2
\end{pmatrix}=\begin{pmatrix}
\psi_2-\dfrac{1}{\eta} A\psi_1-\dfrac{U}{\eta} AN_0\partial_x w_1\\[.25cm]
M_{\alpha} w_2-\dfrac{1}{\eta}
\mathscr{A}w_1+\dfrac{U}{\eta} 
N_0^*A\partial_{x_1}\psi_1.
\end{pmatrix} \in [\mathscr{D}(A^{1/2})]' \times [\mathscr{D}(\mathscr{A}^{1/2})]'.
\end{equation} 
Taking $V = \mathscr D(A^{1/2})\times\mathscr D(\mathscr A^{1/2})$, and considering $\mathscr G: V \to V'$, we obtain that $\mathscr G$ is $m$-monotone and coercive for appropriately chosen $\eta(U)$, and hence a corollary to Minty's theorem \cite[Proposition 1.2.5]{springer} ensures that $\mathscr G$  surjective. Elliptic regularity for $A$ and $\mathscr A$ then provide that $y=(\phi_1,\phi_2; u_1,u_2)^T \in \mathscr D(T_{\alpha})$ (with appropriate estimates), giving the solution to \eqref{rangecondition}.

With generation accomplished on $Y_{\alpha}$ taken with the topology induced by $((\cdot,\cdot))$, we obtain immediate semigroup generation on $Y_{\alpha}$ in the natural norm induced by \eqref{inner1}, and again, via topological equivalence, semigroup generation on $Y_s$.
\end{proof}

\subsection{Locally Lipschitz Perturbation on Finite Time Intervals}
Consider the perturbation operator $\mathscr F: Y_{s} \to Y_{s}$ given by $\mathscr F (y) = \big( 0,~ \mu \phi; ~0, ~M_{\alpha}^{-1}[p_0-f_v(u)]\big)^T$.
Then the abstract system \begin{equation*} \dfrac{dy}{dt}+T_{\alpha}y=\mathscr{F}(y), ~y(0)=y_0 \in 
	 Y_{s} \end{equation*} is equivalent to the main flow-plate system \eqref{flowplate}. As is shown in \cite{springer}, the sharp regularity of the Airy stress function provides the a local Lipschitz property for the von Karman nonlinearity:
\begin{equation}\label{airy-lip}
\| [u_1,v(u_1)]-  [u_2,v(u_2)]\|_{-\delta} \le
C\big(\|u_1\|_{2}^2+\|u\|_{2}\|_2^2\big)\| u_1-  u_2\|_{2-\delta},~~\delta \in [0,2).
\end{equation}
Hence $\mathscr F: Y_{s} \to Y_{s}$ is a locally Lipschitz perturbation. 

Applying the standard perturbation semigroup argument \cite{pazy}, we obtain:
\begin{lemma}\label{nonl-cor} With $T_{\alpha}$ and $\mathscr F$ as above, 
	the equation \begin{equation*}\label{abstractode} \dfrac{dy}{dt}+T_{\alpha}y=\mathscr{F}(y), ~y(0)=y_0 \in  \cD(T_{\alpha} ) 
	 \end{equation*}  has a unique local-in-time strong solution on $[0,t_{\text{max}})$.
	When $y_0 \in Y_{s}$, we have a unique local-in-time $C(0,t_{\text{max}}; Y_{s})$ mild solution.
	In both cases, when $t_{\text{max}}(y_0)< \infty$, we have that $||y(t)||_{Y_{s}}  \to \infty$ as $t \nearrow t_{\text{max}}(y_0)$.
\end{lemma}
Identifying the abstract ODE in \eqref{abstractode} with the flow-plate system in \eqref{flowplate}, we obtain a local-in-time existence and uniqueness result.
\begin{corollary}\label{solutions}
Consider the system in \eqref{flowplate} with $U \in [0,1)$, $\alpha>0$, and $k\ge 0$. Take $p_0 \in L^2(\Omega)$ and $F_0 \in {H^3(\Omega)}$. Then, for $y_0 = (\phi_0,\phi_1;u_0,u_1) \in \mathscr D(T_{\alpha})$ (resp. $Y_{s}$) there exists a unique, local-in-time strong (resp. generalized) solution $(\phi(t), \phi_t(t); u(t), u_t(t))$ as defined in Section \ref{energiessols}. 
\end{corollary}

\begin{remark} Energy methods and the direct estimate
\begin{equation}\label{oneusing*}||\phi(t)||_{L^2(\realsthree_+)} \le ||\phi_0||_{L^2(\realsthree_+)}+\int_0^t||\phi_t(\tau)||_{L^2(\realsthree_+)} d\tau,\end{equation} yield that the solutions in Corollary \ref{solutions} are valid for $t \in [0,T]$ for any $T>0$.  This point will be superseded by the following section.\end{remark}

\subsection{Bounds in Energy Norm $Y$ and Global Solutions}
In this section we remark that in the norm of $Y$ as defined in \eqref{yspace}, solutions are global-in-time bounded. This allows us to extend our result in Corollary \ref{solutions} to be global in the sense of a solution for $t \in [0,\infty)$. Such extension permits the analysis of long-time behavior of solutions.

The bounds in the following proposition are critical to obtaining the global-in-time boundedness mentioned above. 
\begin{proposition}\label{littleest} First, for $\phi \in W_1(\mathbb R_+^3)$:
	\begin{equation}\label{tracebound}||r_{\Omega}\gamma[\phi]||_{L^2(\Omega)} \le C_{\Omega}||\nabla \phi||_{L^2(\realsthree_+)}.\end{equation}
Next, the interactive energy $E_{int}$, as defined \eqref{energies}, is controlled in the following way:
	\begin{equation}\label{intbound} \big|E_{int}(t)\big| \le \delta \|\nabla \phi(t)\|_{\realsthree_+}^2+C(U,\delta)\|u_{x_1}(t)\|_{\Omega}^2, ~~\delta>0. \end{equation}
	Lastly, the nonlinear potential energy provides  control of low frequencies: ~ for any  $\eta,\epsilon > 0 $ there exists $M_{\epsilon,\eta} $ such that
	\begin{equation}\label{l:epsilon}\|u\|^2_{2-\eta} \leq \epsilon [\|\Delta u\|^2  + ||\Delta v(u)||^2 ] + M_{\eta,\epsilon},~~\forall~u \in (H^2\cap H_0^1)(\Omega).\end{equation}
\end{proposition}
\noindent See \cite{springer,webster} for detailed proofs of the above facts; here, we suffice to say that \eqref{tracebound} follows from the Hardy inequality, and from \eqref{tracebound} the estimate \eqref{intbound} follows via Young's; \eqref{l:epsilon} is obtained through a compactness uniqueness argument that exploits superlinearity of $f_v$.

Now, let us define the positive part of the energy $\mathcal E$ as:
\begin{equation}\label{posenergy}
 \mathcal E_*(t)=\dfrac{1}{2}\big[||u_t||_{L_{\alpha}^2 (\Omega)}^2+ ||\Delta u ||^2+\frac{1}{2}||\Delta v(u)||^2+||\phi_t||^2+||\nabla \phi||^2\big]\end{equation}
 Then, via Proposition \ref{littleest}, we obtain control of the unsigned energy by the positive part:
 \begin{lemma}\label{energybound} For generalized solutions to \eqref{flowplate}, there exist positive constants $c,C,$ and $M$ that are positive, and do not depend on the individual trajectory, such that:
	\begin{equation} \label{estim}
	c \mathcal E_*(t)-M_{p_0,F_0} \le \cE(t) \le C \mathcal E_*(t)+M_{p_0,F_0},
	\end{equation} \end{lemma}
The next lemma is the energy identity, as defined through \eqref{energies}:
\begin{lemma}\label{energyident}
Weak (and hence generalized and strong) solutions to \eqref{flowplate} satisfy the energy identity
$$ \mathcal E(t) + k\int_0^t ||u_t||_{L^2_{\alpha}(\Omega)}^2 d\tau = \mathcal E(0).$$
\end{lemma}
Synthesizing all of the above lemmata, we obtain:
\begin{lemma}\label{globalbound} Any weak (and hence generalized or strong) solution to \eqref{flowplate} will satisfy the bound \begin{equation}\label{apr}
	\sup_{t \ge 0} \left\{\|u_t\|_{L^2_{\alpha}(\Omega)}^2+\|\Delta u\|_{\Omega}^2+\|\phi_t\|_{\realsthree_+}^2+\|\nabla \phi\|_{\realsthree_+}^2 \right\}  \leq C\big(\|y_0\|_Y\big)< + \infty. \end{equation} Thus, solutions are (Lyapunov) stable in time in the norm $Y$.
\end{lemma}
An immediate corollary from the energy identity \eqref{energyident} and the above boundedness is the finiteness of the dissipation integral, which is used critically below.
\begin{corollary}\label{dissint} Any solution \eqref{flowplate} satisfying the energy identity with $k> 0$ has the property $$\int_0^{\infty} \|u_t(t)\|_{L_{\alpha}^2(\Omega)}^2 dt \leq  K(||y_0||_Y) < \infty.$$
\end{corollary}

\begin{remark} We note that global-in-time boundedness of solutions cannot be obtained without accounting for nonlinear effects. Also, we note that any generalized solution has the properties: $y(t) \in C([0,T] ; Y_s)$ and $y(t) \in C([0,\infty); Y)$. This is to say, on infinite time intervals we lose control of the quantity $||\phi(t)||_{0}.$\end{remark}

\section{Equilibria Set}\label{equilsec} Before moving on to discuss the long-time behavior of trajectories, it is worthwhile to discuss the equilibrium/stationary solutions. Thus, 
in this section, we consider the stationary points of the dynamics $(S_t,Y)$, i.e., stationary solutions for \eqref{flowplate} corresponding to:
\begin{equation}\label{static}
\begin{cases}
\Delta^2u+f_v(u)=p_0(\xb)+Ur_{\Omega}\gamma[\partial_{x_3} \phi]& \xb \in \Omega\\
u=\Dn u= 0 & \xb \in \Gamma\\
\Delta \phi -U^2 \partial_{x_1}^2\phi=0 & \xb \in \realsthree_+\\
\partial_{x_3} \phi = U[\partial_{x_1} u]_{\text{ext}}  & \xb \in \partial \realsthree_+
\end{cases}
\end{equation}
	Below, $W_2(\realsthree_{+})=\left\{\phi \in L_{loc}^2(\realsthree_+) : D^{\alpha}\phi\in L^2(\realsthree_{+}), ~|\alpha|=1,2 \right\}$, and a weak solution to \eqref{static} is defined as a pair $\left(u,\phi \right)\in H_0^2(\Omega)\times W_1(\realsthree_{+})$ such that 
\[
\langle \Delta {u}, \Delta w\rangle- \langle \left[ {u}, v({u}) + F_0 \right], w \rangle +  U \langle \gamma \left[ {\phi} \right], \partial_{x_1} w \rangle  = \langle p_0, w\rangle 
\]
and 
\[
(\nabla {\phi}, \nabla \psi)_{\mathbb{R}_{+}^3} - U^2 (\partial_{x_1} {\phi}, \partial_{x_1} \psi)_{\mathbb{R}_{+}^3} + U\langle \partial_{x_1} {u}, \gamma [\psi]\rangle_{\Omega} = 0.
\]

We have the following theorem for the stationary problem:
\begin{theorem}\label{statictheorem}
	Suppose $0 \le U <1$ with $p_0 \in L^2(\Omega)$ and $F_0 \in H^3(\Omega)$. Then a {\em weak} solution $\left(u(\xb),\phi(\xb)\right)$ to \eqref{static} exists and satisfies the additional regularity property $$(u,\phi) \in (H^4\cap H_0^2)(\Omega) \times W_2(\realsthree_+).$$ Such solutions correspond to the extremal points of the potential energy functional $$P(u,\phi) = \frac{1}{2}\|\Delta u\|_{\Omega}^2+\Pi(u)+\frac{1}{2}\|\nabla \phi\|_{\realsthree_+}^2-\dfrac{U^2}{2}\|\partial_{x_1}\phi\|_{\realsthree_+}^2 + U\langle \partial_{x_1} u, tr[\phi]\rangle_{\Omega},$$ considered for $(u,\phi) \in H_0^2(\Omega) \times W_1(\realsthree_+)$. \end{theorem}
	We denote by $\mathcal N$ the set of stationary solutions from Theorem \ref{statictheorem}.
In general, $\mathcal N$ has multiple elements.  The reference \cite{springer} provides an example (a choice of $p_0$ and $F_0$) where there are multiple stationary points:
	~let $p_0(x)\equiv 0$ and $F_0=-\beta\left(x_1^2+x_2^2\right).$ Then there exists $\beta>0$ such that for $\beta>\beta_0$ we have at least three solutions: $(0,0)$, $(u_\beta,\phi_\beta)$ and $(-u_\beta,-\phi_\beta)$---see also \cite{ciarlet}. 

 For given loads $F_0$ and $p_0$, the set of stationary solution is generically finite. This is to say that there is an open dense set $\mathcal R \subset L_2(\Omega)\times H^4(\Omega)$ such that  if $(p_0,F_0) \in \mathcal R$ then the corresponding set of stationary solutions $\mathcal N$ is finite. This follows from the Sard-Smale theorem, as shown in \cite[Theorem 1.5.7 and Remark 6.5.11]{springer}.

\section{Flow with Given Neumann Plate Data}
To perform the qualitative analysis below, it will be necessary to consider the flow equation with {\em prescribed} Neumann data. This will allow us to explicitly compute the relevant Neumann-to-Dirichlet type map in terms of the plate dynamics, among other quantities. 

Consider the problem:
\begin{equation}\label{floweq*}
\begin{cases}
(\partial_t+U\partial_{x_1})^2\phi=\Delta \phi & \text{ in }~\mathbb R_+^3 \times (t_0,T)\\
\partial_{x_3} \phi\big|_{x_3=0} = h(\xb,t) & \text{ in }~\mathbb R^2 \times (t_0,T)\\
\phi(t_0)=\phi_0;~~\phi_t(t_0)=\phi_1 & \text{ in }~\mathbb R_+^3
\end{cases}
\end{equation}
We have the following theorem from \cite{b-c-1,springer,miyatake1973}:
\begin{theorem}\label{flowpot}
Assume $U\ge 0$, $U\ne 1$; take $(\phi_0,\phi_1) \in H^1(\realsthree)\times L^2(\realsthree).$ If ~$h \in C\left([t_0,\infty);H^{1/2}(\mathbb R^2)\right)$ then \eqref{floweq*} is well-posed (in the weak sense) with 
$$\phi \in C\left([t_0,\infty);H^1(\realsthree_+)\right),~~\phi_t \in C\left( [t_0,\infty);L^2(\realsthree_+)\right).$$
\end{theorem}
\begin{remark}
In fact, a stronger regularity result is available. Finite energy $H^1(\Omega) \times L^2(\Omega)$ solutions are obtained with $h\in H^{1/3}((0,T) \times \mathbb R^2)$ \cite{redbook,tataru}. 
\end{remark}

\subsection{Flow Decomposition and Properties}

We may decompose the flow problem from \eqref{floweq*} into two pieces corresponding to zero Neumann data, and zero initial data, respectively: ~$\phi^*$ solves \eqref{flowpot} with $h \equiv 0$, and $\phi^{**}$ solves \eqref{flowpot} with $\phi_0=\phi_1\equiv 0$.

In line with our well-posedness result, Corollary \ref{solutions}, we will consider:
\begin{equation}\label{h}
h(\xb, t) \equiv [u_t + U u_{x_1}]_{\text{ext}} \in C([0,T];H^1(\mathbb R^2)).
\end{equation}

\subsection{Point-wise Formulae}

In this section, we will look at $\phi^*$ and $\phi^{**}$ separately and establish results regarding each of them. These results will be used to obtain useful estimates in the next section.

 For the analysis of $\phi^*$ we use the tools developed in \cite{b-c,b-c-1}, namely, the Kirchhoff type representation for the solution  $\phi^*(\xb,t)$
in $\R_+^3$ (see, e.g., \cite[Theorem~6.6.12]{springer}). We  conclude that
if the initial data   $\phi_0$ and $\phi_1$ are  localized in the ball $K_{\rho} \equiv \mathbb R^3_+ \cap B_\rho(\mathbb R^3)$,
then by  finite dependence on the domain of the signal in 3-D  (Huygen's principle),
   one obtains for any $\tilde \rho$ that  $\phi^*(\xb,t)\equiv 0$ for all $\xb\in  K_{\tilde \rho} $
and $t\ge t_{\tilde\rho}$. 
Thus $\phi^*$ tends to zero in the sense of the local flow energy, i.e.,  \begin{equation}\label{starstable} \|\nabla \phi^*(t)\|_{L^2( K_{\tilde\rho} )}^2 + \|\phi^*_t(t)\|_{L^2( K_{\tilde \rho} )} \to 0, ~~ t \to \infty,\end{equation} for all fixed $ \tilde\rho>0$.
Also, in this case,
\begin{equation}\label{traceh}
\big(\partial_t+U\partial_{x_1}\big)\gamma[\phi^*]\equiv0,~~~\xb\in \Omega,~t\ge t_{\tilde \rho}.
\end{equation}

On the other hand, for $\phi^{**}$ as above, we have the following result: 
\begin{theorem}\label{flowformula}
	Let $$\ds h(\xb,t) =[u_t (x_1,x_2,t)+Uu_{x_1}(x_1,x_2,t)]_{\text{ext}},$$ there exists a time $t^*(\Omega,U)$ such that, for all $t>t^*$, we have the following representation for the weak solution:
	\begin{equation}\label{phidef}
	\phi^{**}(\xb,t) = {-}\dfrac{\chi(t-x_3) }{2\pi}\int_{x_3}^{t^*}\int_0^{2\pi}(u^{\dag}_t(\xb,t,s,\theta)+Uu^{\dag}_{x_1}(\xb,t,s,\theta))d\theta ds.
	\end{equation}
	where $\chi(s) $ is the Heaviside function. The time $t^*$ is given by:
	\begin{equation}\label{tstar}
	t^*=\inf \{ t~:~\xb(U,\theta, s) \notin \Omega \text{ for all } (x_1,x_2) \in \Omega, ~\theta \in [0,2\pi], \text{ and } s>t\},
	\end{equation} with ~\begin{equation}\label{xescape}\xb(U,\theta,s) = \left(x_1-(U+\sin \theta)s,x_2-s\cos\theta\right) \subset \realstwo\end{equation} (not to be confused with $\xb = (x_1,x_2)$).
	\end{theorem}

We may use the elementary formula 
	\begin{equation}\label{partialt}
	\partial_th^{\dag}(\xb,t,s,\theta)=-\frac{d}{ds}h^{\dag}(\xb,t,s,\theta)-U\partial_{x_1}h^{\dag}(\xb,t,s,\theta)-\frac{s}{\sqrt{s^2-x_3^2}}[M_{\theta}h^{\dag}](\xb,t,s,\theta)
	\end{equation}
	where, $M_{\theta} = \sin\theta\partial_{x_1}+\cos \theta \partial_{x_2}$ and  \begin{equation}\label{dag} u^{\dag}(\xb, t,s,\theta)=[u]_{ext}\left(x_1-Us+\sqrt{s^2-x_3^2}\sin \theta,x_2-\sqrt{s^2-x_3^2}\cos\theta,t-s\right),\end{equation}
	in computing partials of $\phi^{**}$. For $\phi_t^{**}$ we have:
	\begin{align}\label{phitest}
	\begin{split}
	\phi^{**}_t(\xb,t)= & \frac{1}{2\pi}\Bigg\{\int_0^{2\pi}d\theta u_t^{\dag}(\xb,t,t^*,\theta)-\int_0^{2\pi}d\theta u_t^{\dag}(\xb,t,x_3,\theta)\\ 
	& +\int_{x_3}^{t^*}ds\frac{s}{\sqrt{s^2-x_3^2}}\int_0^{2\pi}d\theta [M_{\theta}u_t^{\dag}](\xb,t,s,\theta)\Bigg\}\\
	\end{split}
	\end{align}
Similarly, the spatial partials for $i=1,2$ are:
\begin{align} \label{partialder}
\phi^{**}_{x_i}(x,t) = ~\frac{1}{2\pi} \int_{x_3}^{t^*}\int_0^{2\pi} [\partial_t+U\partial_{x_1}]u_{x_i}^\dagger(x,t,s,\theta) d\theta ds
=&~\frac{1}{2\pi}\int_{x_3}^{t^*}\int_0^{2\pi} U\partial_{x_1}u_{x_i}^\dagger(x,t,s,\theta) d\theta ds \\
&+\frac{1}{2\pi}\int_{x_3}^{t^*}\int_0^{2\pi} \partial_tu_{x_i}^\dagger(x,t,s,\theta) d\theta ds\nonumber
\end{align}
Differentiation in $x_3$ is direct, and so with cancellation, it yields:
\begin{align} \partial_{x_3}\phi^{**}(x,t)&= (\partial_t+U\partial_{x_1})u(x_1-Ux_3,x_2,t-x_3)\\
&+\frac{1}{2\pi}\int_{x_3}^{t^*}\frac{x_3}{\sqrt{s^2-x_3^2}}\int_0^{2\pi}d\theta \big[(\partial_t+U\partial_{x_1})[M_\theta u]^\dagger\big](x,t,s,\theta) \end{align}

From these direct calculations, bounds on solutions can be obtained directly \cite[Lemma 8]{ryz} using interpolation:
\begin{lemma}\label{compact2}
For \eqref{floweq*}, taken with $h(\xb,t)=(u_t+Uu_x)_{\text{ext}}$, we have
\begin{align}
\|\nabla  \phi^{**}(t)\|^2_{\eta, K_{\rho} }&+\|\phi_t^{**}(t)\|^2_{\eta, K_{\rho} }\nonumber \\
\le  ~C(\rho)& \big\{\|u(\cdot)\|^2_{H^{s+\eta}(t-t^*,t;H_0^{2+\eta}(\Omega))}+\|u_t(\cdot)\|^2_{H^{s+\eta}(t-t^*,t;H_0^{1+\eta}(\Omega))} \big\}
\end{align} for $s,\eta \ge 0$, $0<s+\eta<1/2$ and $t>t^*(U,\Omega)$.
\end{lemma}

From here, for smooth solutions, we can explicitly solve for the needed Dirichlet trace of the material derivative appearing on the RHS of  \eqref{flowplate} in the plate equation in terms of the Neumann data $h=[u_t+Uu_{x_1}]$.
	Considering the term $$r_{\Omega}\gamma\left[\left(\partial_t+U\partial_{x_1}\right) \phi\right]=r_{\Omega}\gamma\left[\left(\partial_t+U\partial_{x_1}\right) \phi^{**}\right]$$ for $t>t_{\rho}$ by \eqref{traceh}, where again $\rho$ corresponds to the $supp(\phi_0),~supp(\phi_1) \subset K_{\rho}$.  Using the above expressions for $\partial_t\phi^{**}$ \eqref{phitest} and $\partial_{x_1}\phi^{**}$ \eqref{partialder}, we obtain
	\begin{equation}\label{delayt}
	r_{\Omega}\left(\partial_t+U\partial_{x_1}\right)\gamma\left[\phi^{**}\right]=-(\partial_t+U\partial_{x_1})u-q(u^t),
	\end{equation}
	for $t\geq \max\{t^*,t_{\rho}\}$ ($t^*$ as defined above in \eqref{tstar}) with
	\begin{equation}\label{qyou}
	q(u^t)=\dfrac{1}{2\pi}\int_0^{t^*}ds\int_0^{2\pi}d\theta \left[M^2_{\theta}\left[u\right]_{ext}(\mathbf x(U,\theta,s) \right],
	\end{equation}
	and ~$\mathbf x(U,\theta,s)$ as in \eqref{xescape}.
	
	The notation above for $u^t$ indicates the entire set $\big\{u(t+s)~:~s \in (-t^*,0)\big\},$ where $t^*$ is the fixed delay time given in \eqref{tstar} depending only on $\Omega$ and $U$; this notation is used in considerations with dynamical systems with delay/memory \cite{springer,Chu92b,oldchueshov1}.

\section{Reduced Plate Dynamics}
\subsection{Plate Reduction Theorem} 
From all of the point-wise formulate of the previous section---including the calculation of the ``Neumann-to-material-derivative-trace in \eqref{qyou}---we obtain the theorem below by waiting a time $t^{\#}=\max\{t^*~,~t_{\rho}\}$.
\begin{theorem}\label{rewrite}
	Let $k \ge 0$,  and $(\phi_0,\phi_1; u_0,u_1)^T \in H_0^2(\Omega) \times L^2(\Omega) \times H^1(\realsthree_+) \times L^2(\realsthree_+)$. Assume that there exists an $\rho$ such that $\phi_0(\xb) = \phi_1(\xb)=0$ for outside $K_\rho$.  Then the there exists a time $t^{\#}(\rho,U,\Omega) > 0$ such that for all $t>t^{\#}$ the plate solution $u(t)$ to (\ref{flowplate}) satisfies the following equation (in a weak sense):
	\begin{equation}\label{reducedplate}
	M_{\alpha}u_{tt}+\Delta^2u+kM_{\alpha}u_t+f_v(u)=p_0-(\partial_t+U\partial_{x_1})u-q(u^t)
	\end{equation}
	with
	\begin{equation}\label{potential*}
	q(u^t)=\dfrac{1}{2\pi}\int_0^{t^*}ds\int_0^{2\pi}d\theta [M^2_{\theta} u_{\text{ext}}](x_1-(U+\sin \theta)s,x_2-s\cos \theta, t-s),
	\end{equation}
with $M_{\theta}$ and $t^*$ as in the previous section.
\end{theorem}

We then have the following direct estimates on the delay potential $q(u^t)$ \cite{b-c,b-c-1,springer,delay}:
\begin{proposition}\label{pr:q}
	Let $q(u^t)$ be given by (\ref{potential*}). Then \begin{equation}\label{qnegest}
	||q(u^t)||^2_{-1} \le ct^*\int_{t-t^*}^t||u(\tau)||^2_1d\tau
	\end{equation} for any $u \in L^2(t-t^*,t;H_0^1(\Omega))$.
	If $u \in L^2_{loc}(-t^*,+\infty ; (H^2\cap H_0^1)(\Omega))$ we also have
	\begin{equation}\label{qnegest3}
	||q(u^t)||^2 \le ct^*\int_{t-t^*}^t||u(\tau)||^2_2d\tau,~~~~~~
	\int_0^t ||q(u^\tau)||^2 d\tau \le c[t^*]^2\int_{-t^*}^t||u(\tau)||^2_2d\tau
	,~~~\forall t\ge0.
	\end{equation}
\end{proposition}

\subsection{Reduced Dynamical System $(T_t,\mathbf H)$} \label{delaysys}
With these estimates, the system given in \eqref{reducedplate}--\eqref{potential*} is independently well-posed \cite{delay} as plate equation with memory \cite{delay,springer}. Specifically, on the space for initial data
$$u(0)=u_0 \in H_0^2(\Omega),~~u_t(0)=u_1 \in H_0^1(\Omega),~~\\ u|_{t \in (-t^*,0)} = \eta\in L^2(-t^*,0;H^2_0(\Omega)).$$
 
 In application, we will consider an initial datum $y_0 \in Y$ corresponding to the full flow-plate dynamics $S_t(y_0)$ in \eqref{flowplate}. We wait a sufficiently long time $t^{\#}(\rho,U,\Omega)$ and employ the reduction result  Theorem \ref{rewrite}, and we may consider the ``initial time" ($t=t_0>t^{\#}$) for the delay dynamics. At such a time, the data which is fed into \eqref{reducedplate} is $x_0=(u(t_0),u_t(t_0),u^{t_0})$, where this data is determined by the full dynamics of \eqref{flowplate} on $(t_0-t^*,t_0)$. Thus, given a trajectory $S_t(y_0)=y(t)=(\phi(t),\phi_t(t);u(t),u_t(t))^T \in Y$, we may analyze the corresponding delay evolution $(T_t,\mathbf H)$, with $\mathbf H \equiv H_0^2(\Omega)\times L^2(\Omega)\times L^2\left(-t^*,0;H_0^2(\Omega)\right),$ with given data $x_0 \in \mathbf H$. We then have that $T_t(x_0)=\left(u(t),u_t(t);u^t\right)$ with $x_0=(u_0,u_1,\eta)$. The natural norm is taken to be $$||(u,v;\eta)||^2_{\mathbf H} \equiv ||\Delta u||^2+||v||^2+\int_{-t^*}^0||\Delta \eta(t+s)||^2 ds.$$

Using standard multiplier methods, along with the a priori boundedness in Lemma \ref{globalbound}, we obtain via Gr\"onwall's inequality the Lipschiz estimate below.
\begin{lemma}\label{l:lip}
Suppose $u^i(t)$ for $i=1,2$ are solutions to (\ref{reducedplate}) with different initial data and $z=u^1-u^2$. Additionally, assume that
\begin{equation}\label{bnd-R}
||u_t^i(t)||_{L^2_{\alpha}(\Omega)}^2+||\Delta u^i(t)||^2 \le R^2, ~i=1,2
\end{equation}
 for some $R>0$ and all $t \in [0,T]$. Then there exists $C>0$ and $a_R\equiv a_R(t^*)>0$  such that
\begin{align}\label{dynsysest}
||z_t(t)||_{L^2_{\alpha}(\Omega)}^2+||\Delta z(t)||^2 \le &~ Ce^{a_Rt}\Big\{||\Delta(u^1_0-u^2_0)||^2+||u^1_1-u^2_1||_1^2+\int_{-t^*}^0||\eta^1(\tau)-\eta^2(\tau)||_2^2d\tau\Big\}
\end{align}
for all $t \in [0,T]$.
\end{lemma}

\section{Smooth Global Attractor for Reduced Plate Dynamics}

The main result in this section is that the plate dynamical system $(T_t,\mathbf H)$ has a compact global attractor which has additional nice properties.

	We recall that (see, e.g., \cite{Babin-Vishik,springer}) for a generic dynamical system $(S_t,H)$,
a compact \textit{global attractor} $A \subset \subset H$ is an invariant set (i.e., $S_tA=A$ for all
$t \in \mathbb R_+$) that uniformly attracts any bounded set $B\subset H$:~~
$\ds \lim_{t\to+\infty}d_{H}\{S_t B~|~A\}=0,$~where
$d_H$ corresponds to the Hausdorff semidistance. As we will see, $A$ will have finite fractal dimension in this case: $\text{dim}_f A < \infty$. The fractal dimension of a set is defined in terms of minimal coverings, \cite{springer,quasi}, and a set with finite fractal dimension can be included as a subset of some higher dimensional Euclidean space.

		\begin{theorem}[Smooth, Finite Dimensional Global Attractor]\label{maintheorem}
Let $k>0$, $U \neq 1$,  $p_0 \in L^2(\Omega)$, and $F_0 \in H^3(\Omega)$ in \eqref{flowplate}. Also assume the flow data $\phi_0,\phi_1 \in Y$ are localized (with supports in $K_\rho$, as in Theorem \ref{rewrite}).  Then the corresponding delay system $(T_t,\bH)$ has a compact global attractor $\mathbf A$ of finite fractal dimension in $\mathbf H$. Moreover, $\mathbf A$ has additional regularity: any full trajectory $y(t)=(u(t),u_t(t),u^t)\subset \mathbf A$, $t\ge 0$, has the property that $u \in C_r(\R;H^3(\Omega)\cap H_0^2(\Omega))$, $u_t \in C_r(\R;H_0^2(\Omega))$, and $u_{tt} \in C_r(\mathbb R; H_0^1(\Omega))$. 
\end{theorem}
\noindent This can be rephrased for the full system $(S_t,Y)$ by projecting on the first two components of $\mathbf H$:
\begin{corollary}\label{th:main2*}
With the same hypotheses as Theorem \ref{maintheorem},  there exists a compact set $\mathscr{U} \subset H_0^2(\Omega) \times H_0^1(\Omega)$ of finite fractal dimension such that for any weak solution $(u,u_t;\phi,\phi_t)$ to (\ref{flowplate}) has  $$\lim_{t\to\infty} d_{Y_{pl}} \big( (u(t),u_t(t)),\mathscr U\big)=\lim_{t \to \infty}\inf_{(w_0,w_1) \in \mathscr U} \big( ||u(t)-w_0||_2^2+||u_t(t)-w_1||_1^2\big)=0.$$ We also have the additional regularity $\mathscr{U} \subset \big(H^3(\Omega)\cap H_0^2(\Omega)\big) \times H_0^2(\Omega)$.
\end{corollary}

The proof of Theorem \eqref{maintheorem} proceeds in two steps: the construction of an absorbing ball by Lyapunov methods, followed by the attainment of the so called quasi-stability property on the absorbing ball. 

\subsection{Construction of Absorbing Ball}
For the non-conservative plate dynamics given by \eqref{reducedplate}, we explicitly construct the absorbing ball via a Lyapunov approach. 
Recalling the definition of $E_{pl}$ from \eqref{energies} and defining the quantities $$\Pi_*(u) = \dfrac{1}{2}\big[||\Delta u ||^2+\frac{1}{2}||\Delta v(u)||^2\big],~~E_*(u,u_t) = \frac{1}{2}||u_t||^2_{L^2_{\alpha}(\Omega)}+\Pi_*(u),$$ we consider the Lyapunov-type function
\begin{align}\label{veedef}
V\big(T_t (x)\big) \equiv &~E_{pl}(u(t),u_t(t))+\nu\big[\langle M_{\alpha}^{1/2}u_t,M^{1/2}_{\alpha}u\rangle +k\langle M_{\alpha}^{1/2}u,M_{\alpha}^{1/2}u\rangle \big]+\mu\int_0^{t^*}\int_{t-s}^t \Pi_*(u(\tau))d\tau ds,
\end{align}
where $T_t(x) \equiv x(t)= (u(t),u_t(t),u^t)$ for $t \ge 0$\footnote{without loss of generality, take $t_0=0$},
 and $\mu, \nu$ are some small, positive numbers to be specified below. Using the elementary inequality
 $$\int_0^{t^*}\int_{t-s}^t\Pi_*(u(\tau)) d\tau ds \le t^*\int_{t-t^*}^t \Pi_*(u(\tau)) d\tau,$$ 
we establish the topological equivalence between $V\big(T_t (x)\big)$ and $E_*$, which is given by the following lemma.
 \begin{lemma} With $(T_t,\mathbf H)$ defined in Section \ref{delaysys}. and $V$ defined as in \eqref{veedef}, 
we have that there exists $\nu_0>0$ such that for all $0< \nu \le \nu_0$ there are $c_0(\nu_0),c_1,c(\nu_0),C >0$
\begin{equation}\label{energybounds}
c_0E_* - c \le V(T_t(x)) \le c_1E_*+\mu C t^*\int_{-t^*}^0 \Pi_*(u(t+\tau))d\tau +c.
\end{equation}
\end{lemma}
A careful but direct calculation of $\dfrac{d}{dt}V(T_t(x))$, coupled with the estimates on the nonlinear potential energy Lemma \ref{l:epsilon} and the estimate on $q(u^t)$ at the $L^2$ level in Lemma \ref{pr:q}, produces, 
for $0<\nu<\min{\{\nu_0,1\}}$, and for $\mu$ sufficiently small, the following lemma:
\begin{lemma}\label{le:48}
For all $k > 0$, there exist $\mu, \nu >0$ sufficiently small, and $c(\mu,\nu,t^*,k),C(\mu,\nu,p_0,F_0)>0$ such that
\begin{align}\label{goodneg}
\dfrac{d}{dt}V(
T_t(x))\le& ~C-c\Big\{E_*(u,u_t) +\int_{-t^*}^0\Pi_*(u(t+\tau))d\tau\Big\}.
\end{align}
\end{lemma}
From this lemma and the upper bound in \eqref{energybounds}, we have a $\delta(k,\mu,\nu)>0$ and a $C(\mu,\nu)$: \begin{equation}\label{gronish}
\dfrac{d}{dt}V(T_t(x)) +\delta V(T_t(x)) \le C,~~t>0.
\end{equation}
The estimate above in (\ref{gronish}) implies (via an integrating factor) that
\begin{equation}\label{balll}
V(T_t(x)) \le V(x)e^{-\delta t}+\dfrac{C}{\delta}(1-e^{-\delta t}).
\end{equation}
Hence, the set
$$
\mathscr{B} \equiv \left\{x \in \bH:~V(x) \le 1+\dfrac{C}{\delta} \right\},
$$  is a bounded forward invariant absorbing set for $(T_t,\mathbf H)$. This, along with \eqref{energybounds}, gives that $(T_t,\bH)$ is ultimately dissipative in the sense of dynamical systems \cite{temam,Babin-Vishik}.

\subsection{Tools from Quasi-stability Theory}

We now proceed by discussing the specific tool we use in the construction of the attractor: quasi-stability \cite{quasi,springer}.
\begin{condition}\label{secondorder} Consider second order (in time) dynamics $(S_t,H)$ where $H=X \times Z$ with $X,Z$ Hilbert, and $X$ compactly embedded into $Z$.  Further, suppose $y= (x,z) \in H$ with $S_ty =(x(t),x_t(t))$ where the function $x \in C(\mathbb R_+,X)\cap C^1(\mathbb R_+,Z)$.
\end{condition}
 Condition \ref{secondorder} restricts our attention to second order, hyperbolic-like evolutions.
\begin{condition}\label{locallylip} Suppose the evolution operator $S_t: H \to H$ is locally Lipschitz, with Lipschitz constant $a(t)\in L^{\infty}_{loc}([0,\infty))$:
\begin{equation}\label{specquasi}
\|S_ty_1-S_ty_2\|_H^2 \le a(t)\|y_1-y_2\|_H^2.
\end{equation}
\end{condition}
\begin{definition}\label{quasidef}
With Conditions \ref{secondorder} and \ref{locallylip} in force, suppose that the dynamics $(S_t,H)$ admit the following estimate for $y_1,y_2 \in B \subset H$:
\begin{equation}\label{specquasi*}
\|S_ty_1-S_ty_2\|_H^2 \le e^{-\gamma t}\|y_1-y_2\|_H^2+C_q\sup_{\tau \in [0,t]} \|x_1-x_2\|^2_{Z_*}, ~~\text{ for some }~~\gamma, C_q>0,
\end{equation} where $Z \subseteq Z_* \subset X$, and the last embedding is compact. Then we say that $(S_t,H)$ is {\em quasi-stable} on $B$.
\end{definition}

We now run through a handful of consequences of the type of quasi-stability described by Definition \ref{quasidef} above for dynamical systems $(S_t,H)$ satisfying Condition \ref{secondorder}
\cite[Proposition 7.9.4]{springer}.
\begin{theorem}\label{doy}
If a dynamical system $(S_t,H)$ satisfying Conditions  \ref{secondorder} and \ref{locallylip} is quasi-stable on an absorbing ball $ B \subset H$, then there exists a compact global attractor $\mathbf A \subset \subset H$.
$$\sup_{t \in \mathbb R} \left\{\|x_t(t)\|^2_X+\|x_{tt}(t)\|_Z^2\right\} \le C,$$
where the constant $C$ above depends on the ``compactness constant" $C_q$ in \eqref{specquasi*}.
\end{theorem}
\noindent Elliptic regularity can then be applied to the equation itself generating the dynamics $(S_t,H)$ to recover regularity for $x(t)$ in a norm higher than that of the state space $X$.

\subsection{Quasi-stability on the Absorbing Ball}
For the discussion of quasi-stability, we begin with the standard observability and energy inequalities which follow from energy methods developed for the wave equation. The details presented for $\alpha=0$ in \cite{delay} are unaltered different here. 

Let us utilize the notation that ~$E_z(t) = ||z_t(t)||_{L^2_{\alpha}(\Omega)}^2+||\Delta z(t)||^2$. We state the following estimates without proof.
\begin{lemma}[Preliminary Estimates]\label{le:observbl}
Let $$u^i \in C(0,T;H_0^2(\Omega))\cap C^1(0,T;L^2_{\alpha}(\Omega)) \cap L^2(-t^*,T;H_0^2(\Omega))$$ solve (\ref{reducedplate}) with appropriate initial conditions on $[0,T]$ for $i=1,2$, $T\ge 2t^*$. Additionally, assume $(u^i(t),u^i_t(t)) \in  B_R(Y_{pl})$ for all $t\in [0,T]$.
Then the following  estimates on $z$ holds for some $\delta \in (0,2]$:
\begin{align}\label{zenergyprelim-3}
\Ez(t)+\int_s^t E_z d\tau \le& ~ a_0\left(\Ez(s)+\int_{s-t^*}^s ||z(\tau)||_{2-\delta}^2 d\tau\right)
  +C(T,R,\delta)\sup_{\tau \in [s,t]}||z||^2_{2-\delta}
\\\nonumber &-a_1\int_s^t \langle f(u^1)-f(u^2),z_t\rangle d\tau.
\end{align}

\begin{align}\label{enest1}
 \frac{T}2 \Big[\Ez(T)+ \int_{T-t^*}^T \Ez(\tau) d\tau\Big] \le& ~ a_2\left(\Ez(0)+\int_{-t^*}^0 ||z(\tau)||_{2}^2 d\tau\right) +C(T,R,\delta)\sup_{\tau \in [0,T]}||z||^2_{2-\delta} \\\nonumber
\\\nonumber &-a_3\int_0^T \int_s^T \langle f(u^1)-f(u^2),z_t\rangle ~d\tau ds -a_4\int_0^T \langle f(u^1)-f(u^2),z_t\rangle d\tau,
\end{align}
 with the positive $a_i$ independent of $T$ and $R$.
\end{lemma}
We note the elementary bound
$$|\langle f(u^1)-f(u^2),z_t\rangle | \le C_\epsilon ||f(u^1)-f(u^2)||_{-\delta}^2+\epsilon||z_t||^2_{\delta}.$$
Utilizing the above bound directly, and invoking the locally Lipschitz nature of the von Karman nonlinearity in \eqref{airy-lip}, we see (rescaling constants) that
\begin{equation}\label{diffest}
\int_s^t\big|\langle f(u^1)-f(u^2) , z_t \rangle\big| d\tau \le \epsilon \int_s^t||z_t||_1^2+C(\epsilon,R,\delta,|t-s|)\sup_{[s,t]}||z(\tau)||^2_{2-\delta}.
\end{equation} 

This yields the estimate from which the quasi-stability property of $(T_t,\mathbf H)$ can be deduced.
\begin{lemma}\label{prequasi}
	Suppose $z=u^1-u^2$ as before, with $y^i(t)=(u^i(t),u_t(t)^i,u^{t,i})$ and $y^i(t) \in \mathscr B$ (i.e., the trajectories lie in the absorbing ball) for all $t\ge 0$. Also, let $\delta >0$ and $\Ez(t)$ be defined as  above.
	Then there exists a time $T$ such that the following estimate holds:
	\begin{equation}\label{prequasi*}
	E_{z}(T)+\int_{T-t^*}^{T}||z(\tau)||_2^2 d\tau\leq \beta \big(E_z(0)+\int_{-t^*}^0||z(\tau)||_2^2d\tau)+C(R,T,t^*,\delta)\underset{\tau \in \lbrack0,T]}{\sup }||z(\tau )||_{2-\delta }^2
	\end{equation}
	with $\beta<1$.
\end{lemma}
\begin{proof}[Proof of Lemma \ref{prequasi}]
Applying \eqref{diffest} to \eqref{zenergyprelim-3} and \eqref{enest1}  with $s=0$ and $t=T$, we obtain

\begin{align}\label{enest2}
\Ez(T)+\int_0^T E_z d\tau \le& ~ a_0\left(\Ez(0)+\int_{-t^*}^0 ||z(\tau)||_{2-\eta}^2 d\tau\right)
+C(T,R,\eta,\epsilon)\sup_{\tau \in [s,t]}||z||^2_{2-\eta}
\\\nonumber &+\epsilon a_1\int_0^T \|z_t(\tau)\|^2_1d\tau.
\end{align}
and
\begin{align}\label{enest3}
\frac{T}2 \Big[\Ez(T)+ \int_{T-t^*}^T \Ez(\tau) d\tau\Big] \le& ~ a_2\left(\Ez(0)+\int_{-t^*}^0 ||z(\tau)||_{2}^2 d\tau\right) +C(T,R,\eta,\epsilon)\sup_{\tau \in [0,T]}||z||^2_{2-\eta}
\\\nonumber &+C T\epsilon\int_0^T \|z_t(\tau)\|^2_1d\tau
\end{align}
for $T\geq \max\{t^*,1\}$.
After adding  \eqref{enest2} and \eqref{enest3} and invoking the Sobolev embeddings, we can drop suitable terms to obtain
\begin{align}\label{enest5}
\frac{T}2 \Ez(T)&+\left[\left\{\alpha c_p -\epsilon\left(a_1+CT\right)\right\}\int_0^T \|z_t(\tau)\|^2_1d\tau\right]\\ \nonumber&+ \frac{cT}2\int_{T-t^*}^T \|z(\tau)\|_2^2 d\tau\leq A\left(\Ez(0)+\int_{-t^*}^0 ||z(\tau)||_{2-\delta}^2 d\tau\right)+C(T,R,\delta,\epsilon)\sup_{\tau \in [0,T]}||z||^2_{2-\delta}
\end{align}
where $0<c<1$ and  $c_p$ is a Poincare constant. Scaling $\epsilon$ small enough, and $T$ large enough, we obtain after simplifying
\begin{equation}\label{prequasi**}
E_{z}(T)+\int_{T-t^*}^{T}||z(\tau)||_2^2 d\tau\leq \beta \Big(E_z(0)+\int_{-t^*}^0||z(\tau)||_2^2d\tau\Big)+C(R,T,t^*,\delta)\underset{\tau \in \lbrack0,T]}{\sup }||z(\tau )||_{2-\delta }^2
\end{equation}
with $\beta<1$.
\end{proof}

We note that the necessary  Lipschitz estimate in \eqref{specquasi} holds by Lemma \ref{l:lip}.
Via the semigroup property for the evolution $(T_t,\mathbf H)$ (iterating on intervals of size $T$), we obtain from \eqref{prequasi**} in a standard way the quasi-stability estimate in \eqref{specquasi*}. Hence the dynamical system $(T_t,\mathbf H)$ is {\em quasi-stable} on the absorbing ball $\mathscr B$ in the sense of Theorem \ref{doy}.

\subsection{Attractor and Its Properties}
We conclude this section by pointing out the existence and discussing the regularity of the global attractor for the decoupled plate sysytem. We first observe that Theorem~\ref{doy} applies to the dynamial system $(T_t,\mathbf H)$. Hence, we have existence of the global attractor $\mathbf A\subset\subset \mathbf H$ for the system $(T_t,\mathbf H)$ and the estimate:
\begin{equation}
\sup_{t \in \mathbb R} \left\{\|u_t(t)\|^2_{H_0^2(\Omega)}+\|u_{tt}(t)\|_{H_0^1(\Omega)}^2\right\} \le C.
\end{equation}
Consequently, $u_t(t)\in H_0^2(\Omega)$ and $u_{tt}(t)\in H_0^1(\Omega)$ for each $t\in \mathbb R$. Applying elliptic regularity to the equation $$\Delta^2u = -M_{\alpha}u_{tt}+kM_{\alpha}u_t+[u,v(u)+F_0]+p_0-u_t-Uu_x-q(u^t) \in H^{-1}(\Omega)$$ point-wisedly in time, we obtain that $u(t)\in H^3(\Omega)\cap H_0^2(\Omega).$ Therefore, $\mathbf A$ possesses better regularity than $\mathbf H$ and we obtain the statement of Theorem~\ref{maintheorem}.
\section{Stabilization to Equilibria Set}
Finally, in this section, we show that in the subsonic case, when {\em any amount of damping is present} $k>0$, that trajectories stabilize to the equilibrium set $\mathcal N$ as in Section \ref{equilsec}. From an applied point of view, this means that for subsonic flows, physical panels do not experience aerodynamic instability. {\em Said differently: there is no subsonic panel flutter.} The main theorem proved in this section is presented below.
\begin{theorem}\label{unibound}
	Let $\alpha>0,$ $0\leq U<1,$ $k>0,$ and the assumptions of Theorem~\ref{rewrite} be in force. Then for any weak solution $(u(t);\phi(t))$, with flow initial data $\phi_0$ and $\phi_1$ localized to $K_{\rho}$, we have
	\begin{equation}
	\lim_{t\rightarrow\infty}\inf_{\{\bar{u};\bar{\phi}\}\in\mathcal{N}}\Big\{||u(t)-\bar{u}(t)||_{2}^2+ ||u_t(t)||_{2}^2+||\phi(t)-\bar{\phi}(t)||_{H^1( K_{\tilde\rho} )}^2+||\phi_t(t)||_{L^2( K_{\tilde\rho} )}^2\Big\}=0
	\end{equation}
	for any $\tilde \rho>0$.
\end{theorem}
This theorem is discussed concisely in \cite{springer,chuey}. 

By the {\em generic finiteness} of the stationary set discussed in Section \ref{equilsec}, for {\em most} loads $p_0$ and $F_0$, the set $\mathcal N$ is finite. In which case, the equilibria set corresponding to Theorem \ref{statictheorem} is discrete and isolated. When this occurs, we can improve the result above. 
\begin{corollary}\label{improve} Assume that $\mathcal N$ is an isolated set.	
	Let the hypotheses of Theorem \ref{unibound} (strong convegence) be in force; then for any generalized solution $(u,\phi)$ to \eqref{flowplate} there exists a stationary point $(\bar{u}, \bar{\phi})$ satisfying \eqref{static} such that
	\begin{align*}\lim_{t \to \infty} \left\{\|u(t)-\bar{u}\|^2_{2}+\|u_t(t)\|^2+\|\phi(t)-\bar{\phi}\|_{H^1( K_{\tilde \rho} )}^2+\|\phi_t(t)\|^2_{L^2( K_{\tilde \rho} )} \right\}=0, 
	\end{align*}
	for any $\tilde \rho>0.$
	
\end{corollary}

The proof of this theorem proceeds in steps, utilizing first the finiteness of the dissipation integral in Lemma \ref{dissint} (from uniform boundedness of trajectories in the norm $Y$ in Lemma \ref{globalbound}). These facts along with the compactness of the attractor for $(T_t,\mathbf H)$ allow us to conclude strong convergences for the plate. We then ``transfer" these strong convergences to the flow via the Neumann lift corresponding to the $\phi^{**}$ in Theorem \ref{flowformula} and the resulting Neumann-to-Dirichlet expression for the material flow derivative \eqref{potential*}.

\subsection{Step 1: Plate Convergences}
\begin{proposition}\label{Hneg} Let $u$ be a generalized solution to \eqref{flowplate}. Then $M_{\alpha}u(t) \to 0$ in $H^{-1}(\Omega)$ as $t \to \infty$.\end{proposition}
\begin{proof}[Proof of Proposition \ref{Hneg}]
Multiplying the plate equation in (\ref{flowplate}) by a test function $w\in C_0^{\infty}(\Omega)$, we obtain
\begin{equation}\label{longone}
\langle M_{\alpha}u_{tt},w\rangle=-\langle\Delta u,\Delta w\rangle+\langle p_0,w\rangle-k\langle M_{\alpha}u_t,w\rangle+\langle[u(t),v+F_0],w\rangle+\langle r_{\Omega}\big[\gamma(\phi_t+U\phi_x)\big],w\rangle.
\end{equation}
The first and second terms on the RHS of the equality are uniformly bounded in time by Lemma~\ref{globalbound}. 

For the third term, we have
\begin{align}\label{vanilla}
\begin{split}
|k\langle M_{\alpha}u_t,w\rangle|&\leq k(||u_t||\cdot||w||) + k\alpha(||\nabla u_t||\cdot||\nabla w||).
\end{split}
\end{align}
For the fourth term, Theorem $1.4.3$ and Corollary $1.4.5$ from \cite{springer} provides us with the following estimates for $f_v$:
\begin{align}\label{vanilla1}
\begin{split}
|\langle[u,v+F_0],w\rangle|&\leq \|[u,v]\|\|w\|+\|[u,F_0]\|_{-1}\|w\|_1\\
&\leq C_1\|u\|_2^3\|w\|+C_2\|u\|_2\|F_0\|_2\|w\|_1
\end{split}
\end{align}
To estimate the fifth term above in \eqref{longone}, recall that
\[r_{\Omega}\left(\partial_t+U\partial_{x_1}\right)\gamma\left[\phi(t)\right]=-(\partial_t+U\partial_{x_1})u(t)-q(u^t),
\]
for $t\geq t^*$, where $t^*$ and $q(u^t)$ are defined as in Theorem~\ref{rewrite}.
Also,
\begin{align}\label{vanilla2}
\begin{split}
\left|\left\langle r_{\Omega}\big[\gamma(\phi_t(t)+U\phi_x(t))\big],w\right\rangle\right|&\leq \|r_{\Omega}\big[\gamma(\phi_t(t)+U\phi_x(t))\big]\|\cdot\|w\|\\
&\leq\left(\|u_t(t)\|+U\|u_x(t)\|+\|q(u^t)\|\right)\|w\|\\
&\leq C\left(\|u_t(t)\|_{L^2_{\alpha}}+U\|u(t)\|_2\right)+ct^*\int_{t-t^*}^t||u(\tau)||_2d\tau \text~~{(by~ \eqref{qnegest3})}.
\end{split}
\end{align}

The right hand sides of  \eqref{vanilla}, \eqref{vanilla1} and \eqref{vanilla2} are then uniformly bounded in time by Lemma~\ref{globalbound}. Hence, by\eqref{longone},  $\left|(M_{\alpha}u_{tt},w)\right|=\left|\partial_t(M_{\alpha}u_{t},w)\right|$ is uniformly bounded in time for each $w\in C_0^{\infty}(\Omega)$.\\
Next we see that
\begin{align}\label{vogue}
\begin{split}
\int_0^{\infty}\Big|\langle M_{\alpha}u_t,w\rangle\Big|^2d\tau\ &\leq\int_0^{\infty}\Big|\langle u_t,w\rangle\Big|^2d\tau+\alpha\int_0^{\infty}\Big|\langle\nabla u_t,\nabla w\rangle\Big|^2d\tau\\
&\leq C(||w||_1)\int_0^{\infty}||u_t||_{1,\Omega}^2< \infty
\end{split}
\end{align}
by the finiteness of the dissipation integral given in Corollary~\ref{dissint}. 
Hence by Barbalat's lemma and density\footnote{Barbalat's Lemma : Suppose $f(t) \in C^1(a, \infty)$ and $\displaystyle\lim_{t\to \infty} f(t) =\alpha < \infty.$ If $f'(t)$ is uniformly continuous, then $\displaystyle\lim_{t\to\infty} f'(t) = 0.$ In our case, we take $f(t)=\int_0^{t}\Big|(M_{\alpha}u_t,w)\Big|^2d\tau$}, we obtain
\begin{equation}\label{barb}
\lim_{t\rightarrow\infty}\langle M_{\alpha}u_t(t),w\rangle=0,~\forall w \in H_0^1(\Omega).
\end{equation}
\end{proof}

\begin{proposition}\label{key2}
	Let $\mathscr{U}$ be as in Theorem~\ref{th:main2*} and $u$ be a generalized solution to \eqref{flowplate}. Then, given any sequence of moments $t_n \to \infty$, there is a subsequence $\{t_{n_l}\}$ and $(\bar{u},\Tilde{u})\in\mathscr U\subset H_0^2(\Omega)\times H_0^1(\Omega)$ (depending on $\{n_l\}$) such that
	\[
	\lim_{l\rightarrow \infty}||u(t_{n_l})-\bar{u} ||^2_{2}= \lim_{l\rightarrow \infty}||u_t(t_{n_l})-\Tilde{u} ||^2_{1}=0
	\] \end{proposition}
	\begin{proof}[Proof of Proposition \ref{key2}]
		Consider some $t_n\to\infty$ and define for $t_1$:  $$\displaystyle \inf_{(v_0;v_1)\in\mathscr{U}}||(u(t_1);u_t(t_1))-(v_0;v_1)||_{2,\Omega}^2=\rho.$$ Since $(u(t_1),u_t(t_1))\in Y_{pl}$
		and from Corollary\eqref{th:main2*} we have that
		\[
		\lim_{t\to\infty} d_{Y_{pl}} \big( (u(t),u_t(t));\mathscr U\big)=0,
		\]
		 there is a sequence $\{(v_{0,n}^1,v_{1,n}^1)\}\subset \mathscr{U}$ such that
		\[
		\lim_{n\rightarrow\infty}||(u(t_1),u_t(t_1))-(v_{0,n}^1,v_{1,n}^1)||_{Y_{pl}}=\rho.
		\]
		Since $\mathscr U$ is compact, there is a subsequence $\{(v_{0,n_k}^1;v_{1,n_k})\}$ so that
		$$ \displaystyle\lim_{k\rightarrow\infty}(v_{0,n_k}^1,v_{1,n_k})=(v_0^1,v_1^1)\in \mathscr{U}.
		$$  As a subsequence of a convergence sequence, we thus have
		\[
		||(u(t_1),u_t(t_1))-(v_{0}^1,v_{1}^1)||_{Y_{pl}}=\lim_{k\rightarrow\infty}||(u(t_1),u_t(t_1))-(v_{0,n_k}^1,v_{1,n_k}^1)||_{Y_{pl}}=\rho
		\]
		Iterating this procedure, to each $(u(t_n);u_t(t_n))\in{\mathscr{S}}$ we can find a $(v_0^n,v_1^n)\in\mathscr U$ such that
		\[
		d_{Y_{pl}}\Big\{(u(t_n),u_t(t_n)),\mathscr U\Big\}= ||(u(t_n),u_t(t_n))-(v_{0}^n,v_{1}^n)||_{Y_{pl}}
		\]
		Then, from this sequence $\{(v_{0}^n,v_{1}^n)\} \subset \mathscr U$, via compactness of $\mathscr U \subset Y_{pl}$, we  obtain a convergent subsequence and an element element $(v_{0},v_{1})\in \mathscr U$. Hence:
		\[
		\lim_{l\rightarrow\infty}||(u(t_{n_l}),u_t(t_{n_l}))-(v_{0},v_{1})||_{Y_{pl}}=\lim_{l\rightarrow\infty}d_{Y_{pl}}\Big\{(u(t_{n_l}),u_t(t_{n_l})),\mathscr U\Big\}=0.
		\]	
	\end{proof}
	Combining this strong subsequential convergence with the fact in Proposition \ref{Hneg} that ~$M_{\alpha}u_t(t) \to 0~\text{ in }~~H^{-1}(\Omega),$ we obtain immediately:
	\begin{proposition}\label{u_t}
		Let $u$ be a generalized solution to \eqref{flowplate}. Then $$\lim_{t \to \infty} ||u_t(t)||_1=\lim_{t \to \infty}||M_{\alpha}^{1/2}u_t|| =0.$$
	\end{proposition}

Now, given a sequence of moments in time that converges to  infinity,  from \eqref{key2} we have that there exists a subsequence $\{t_n\}$ and a point $\bar u \in H_0^2(\Omega)$ such that
\begin{equation}\label{conv5}
||u(t_n)-\bar{u}||_{2}^2\rightarrow 0, \text{ as } n\rightarrow \infty.
\end{equation}
 Moreover, because
\[
||u(t_n+\tau)-u(t_n)||_1\leq\int_{t_n}^{t_n+\tau}||u_t(s)||_1ds\leq\tau  \max_{s\in [t_n,t_n+\tau]}||u_t(s)||_{1},
\] 
and $u_t \in C(H_0^1(\Omega))$ with $u_t(t) \to 0$ in $H_0^1(\Omega)$,
we conclude that 
\begin{equation}\label{limit1}
\max_{\tau\in [-a,a]}||u(t_n+\tau)-\bar{u}||_{1}\rightarrow 0 \text{ as } n\rightarrow 0
\end{equation}
for every finite $a>0$. By interpolation, for $\delta>0$, we obtain
\begin{equation}\label{limit2}
||u(t_n)-\bar{u}||_{2-\delta}\leq ||u(t_n)-\bar{u}||_{2}^{1-\delta}\hspace{2mm}||u(t_n)-\bar{u}||_{1}^{\delta}.
\end{equation}
Since $\|u(t_n)-\bar{u}\|_{2}$ is bounded, from equations \ref{limit1} and \ref{limit2}, we have
\begin{equation}\label{max}
\max_{\tau\in [-a,a]}||u(t_n+\tau)-\bar{u}||_{2-\delta,\Omega}\rightarrow 0 \text{ as } n\rightarrow \infty
\end{equation}

Using a simple contradiction argument, we can push the Sobolev index to 2.
\begin{proposition}\label{umaxH2}
	Let $u$ be a generalized solution to \eqref{flowplate} and let  $t_n \to \infty$. Then $$\displaystyle \max_{\tau\in [-a,a]}||u(t_n+\tau)-\bar{u}||_{2}\rightarrow 0 \text{ as } n\rightarrow \infty.$$
\end{proposition} 
\begin{proof}[Proof of Proposition \ref{umaxH2}]
	 For any fixed $a$, consider a subsequence $$\left\{\displaystyle \max_{\tau\in [-a,a]}||u(t_{n_k}+\tau)||_{2}\right\}_{k=1}^{\infty}.$$ Since $t \mapsto~ \|u(t)\|_{2}$ is continuous and $[-a,a]$ is compact, $$\max_{\tau\in [-a,a]}||u(t_{n_k}+\tau)||_{2}=||u(t_{n_k}+\tau_k)||_{2}$$ for some $\tau_k\in [-a,a]$. From Corollary\eqref{th:main2*}, we know that the sequence $\{u\left(t_{n_k}+\tau_k\right)\}_{k=1}^{\infty}$ has a convergent subsequence $u(t_{n_{k_l}}+\tau_{k_l})\rightarrow \Tilde{u}\in H_{0}^2(\Omega),$ as well as in any lower Sobolev space;	then 
	\begin{equation}\label{maxlim0}
	\Big|\|u(t_{n_k}+\tau_k)-\bar{u}\|_{2-\delta}-\|\bar{u}-\Tilde{u}\|_{2-\delta}\Big|\leq \|u(t_{n_k}+\tau_k)-\Tilde{u}\|_{2-\delta}\rightarrow 0
	\end{equation}
	From \eqref{max} we know that
	\begin{equation}\label{maxlim}
	\|u(t_{n_k}+\tau_k)-\bar{u}\|_{2-\delta}\leq \max_{\tau\in [-a,a]}||u(t_n+\tau)-\bar{u}||_{2-\delta}\rightarrow 0
	\end{equation}
	Substituting \eqref{maxlim} in \eqref{maxlim0}, we see ~
$	\|\bar{u}-\Tilde{u}\|_{2-\delta}=0$ and we identify the limits.

	Hence, any subsequence of $\big\{\displaystyle \max_{\tau\in [-a,a]}||u(t_n+\tau)-\bar{u}||_{2}\big\}$ has a further convergent subsequence that converges to $0$, yielding the result.
\end{proof}
\noindent
\subsection{Step 2: Lifting to Flow Convergences}
Let $\displaystyle\bar{\phi}(x,t)=\frac{1}{2\pi}\int_{x_3}^{t^*} ds\int_{0}^{2\pi}d\theta U\partial_{x_1}\bar{u}^\dagger(x,t,s ,\theta).$

For $t>t_{\rho}$, $\phi=\phi^{**}$, thus we can replace $\phi$ with $\phi^{**}$ in \eqref{phidef}; we then have the following estimates:
\begin{align}\label{phest1}
\begin{split}
\int_{K_{\rho}}dx|\phi(x,t_n)-\bar{\phi}(x,t_n)|&=\frac{1}{2\pi}\int_{K_{\rho}}dx\Bigg|\int_{x_3}^{t^*} ds\int_{0}^{2\pi}d\theta (\partial_t+U\partial_{x_1} )(u^\dagger(x,t_n,s ,\theta)-\bar{u}^\dagger(x,t_n,s ,\theta))\Bigg|\\
&\leq  t^*\Bigg(\max_{\tau>t-t^*} \|u_t(\tau)\|_{1}+U\max_{\tau\in [-t^*,t^*]}\|u(t_n+\tau)-\bar{u}\|_{1}\Bigg)\int_{K_{\rho}}dx
\end{split}
\end{align}
For $j=1,2$:
\begin{align}\label{phest2}
\begin{split}
\int_{K_{\rho}}dx|\partial_{x_j}\phi(x,t_n)-\partial_{x_j}\bar{\phi}&(x,t_n)|=\frac{1}{2\pi}\int_{K_{\rho}}dx\Bigg|\int_{x_3}^{t^*} ds\int_{0}^{2\pi}d\theta \partial_{x_j}\Big((\partial_t+U\partial_{x_1} )(u-\bar{u})\Big)^\dagger(x,t_n,s ,\theta)\Bigg|\\
&\leq  t^*\Bigg(\max_{\tau>t-t^*} \|u_t(\tau)\|_{1}+U\max_{\tau\in [-t^*,t^*]}\|u(t_n+\tau)-\bar{u}\|_{2}\Bigg)\int_{K_{\rho}}dx
\end{split}
\end{align}
and
\begin{align}\label{phest3}
\begin{split}
\int_{K_{\rho}}dx|\partial_{x_3}\phi(x,&t_n)-\partial_{x_3}\bar{\phi}(x,t_n)|\leq \int_{K_{\rho}}dx|(u-\bar{u})(x_1-Ux_3,x_2,t_n-x_3)|\\
&+\frac{1}{2\pi}\int_{K_{\rho}}dx\Bigg|\int_{x_3}^{t^*}\frac{x_3}{\sqrt{s^2-x_3^2}}\int_0^{2\pi}d\theta [M_{\theta}((\partial_t+U\partial_{x_1})(u-\bar{u}))^\dagger](x,t_n,s,\theta)\Bigg|\\
&\leq \Bigg(\|u-\bar{u}\|_{1}+\max_{\tau>t-t^*} \|u_t(\tau)\|_{1}+U\max_{\tau\in [-t^*,t^*]}\|u(t_n+\tau)-\bar{u}\|_{2}\Bigg)\int_{K_{\rho}}dx.
\end{split}
\end{align}
Also, from equation\eqref{phitest}, we obtain
\begin{align}\label{phest4}
\begin{split}
\|\phi_t(x,t)\|_{K_{\rho}} &\leq\max_{\tau>t-t^*} \|u_t(\tau)\|_{1}\left\{2+t^*\right\}\\
\end{split}
\end{align}
By Proposition~\ref{u_t} and \ref{umaxH2}, all terms on the right hand side of estimates \eqref{phest1}--\eqref{phest4} approach zero. Hence, we obtain the convergence
\[
\|\phi(t_n)-\bar{\phi}(t_n)\|_{1,K_{\rho}}+\|\phi_t(t_n)\|_{K_{\rho}}\rightarrow 0,~~ n\to\infty
\]

\subsection{Step 3: Weak Solution}
In this step, we show that $(\bar{u}; \bar{\phi})$, as constructed in the previous Steps, is a weak solution to \eqref{static}.   
We multiply the plate equation in \eqref{flowplate} by smooth function $w \in C_0^{\infty}(\Omega)$ in $L^2(\Omega)$, integrate from $t_n$ to $t_n + a$ for some $a>0,$ and integrate by parts to obtain
\begin{align*}
\begin{split}
&\langle u_{t},w\rangle\Big|_{t_n}^{t_n+a}  +\int_{t_n}^{t_n+a}\langle\Delta u,\Delta w\rangle+\langle\alpha\nabla u_{t},\nabla w\rangle\Big|_{t_n}^{t_n+a} +k\int_{t_n}^{t_n+a}\langle u_t(t),w\rangle dt\\
&+k\alpha\int_{t_n}^{t_n+a}\langle\nabla u_t,\nabla w\rangle dt-\int_{t_n}^{t_n+a}\langle[u,v+F_0],w\rangle+\int_{t_n}^{t_n+a}\langle p_0,w\rangle dt\\
&-\int_{t_n}^{t_n+a}\langle r_{\Omega}\big[\gamma(\phi_t+U\phi_{x_1})\big],w\rangle dt =0
\end{split}
\end{align*}
Each term may be estimated:
\begin{align}\label{set1}
\begin{split}
\Big |\int_{t_n}^{t_n+a}\langle\Delta u-\Delta\bar{u},\Delta w\rangle dt\Big|&\leq a\|\Delta w\|\max_{\tau\in[0,a]}\|u(t_n+\tau)-\bar{u}(\tau)\|_{2}\\
k\int_{t_n}^{t_n+a}\Big |\langle u_t(t),w\rangle\Big | dt&\leq a\|w\| \max_{\tau\in [0,a]}\|u_t(t_n+\tau)\|_{1}\\
\left|\langle u_{t},w\rangle\Big|_{t_n}^{t_n+a}\right|&\leq  2\|w\|\max_{\tau>t-t^*} \|u_t(\tau)\|_{1}\\
\left|\langle\alpha\nabla u_{t},\nabla w\rangle\Big|_{t_n}^{t_n+a}\right|&\leq 2\alpha\|w\|\max_{\tau>t-t^*} \|u_t(\tau)\|_{1}\\
k\alpha\int_{t_n}^{t_n+a}\left|\langle\nabla u_t,\nabla w\rangle\right|dt&\leq k\alpha a\|\nabla w\|\max_{\tau\in [t_n,t_n+a]}\|u_t(\tau)\|_{1}\\
\end{split}
\end{align}
We substitute $x_3=0$ in \eqref{partialder} and \eqref{phitest} to obtain the pointwise expression for $\langle r_{\Omega}\big[\gamma\langle\phi_t+U\phi_{x_1}\rangle\big],w\rangle$, which is then used to obtain the following estimate:
\begin{align}\label{set2}
\begin{split}
\int_{t_n}^{t_n+a}\left|\langle r_{\Omega}\big[\gamma\langle\phi_t+U\phi_{x_1}\rangle\big],w\rangle dt-U \langle \gamma \left[ \partial_{x_1}\bar{\phi} \right], w \rangle\right| \leq &\left(t^*+2+Ut^*\right)\max_{\tau>t-t^*} \|u_t(\tau)\|_{1}\\
&+U^2t^*\max_{\tau\in[-t^*,a]}\|u(t_n+\tau)-\bar{u}\|_2
\end{split}
\end{align}
As we have noted,  $f_v(u)=[u,v(u)+F_0]$ is locally lipschitz for each $F_0\in H^3(\Omega),$ yielding
\begin{align}\label{converge}
\begin{split}
\left|\langle[u,v(u)+F_0]-[\bar{u},v(\bar{u})+F_0],w\rangle\right|& \leq \|\langle[u,v(u)+F_0]-[\bar{u},v(\bar{u})+F_0]\|\cdot\|w\|\\
& \leq C\big(\|u\|,\|w\|_2,\|F_0\|\big)\|u-\bar{u}\|_2
\end{split}
\end{align}
Each term on the right hand side of \eqref{set1}--\eqref{converge} goes to zero as $n\to \infty$ by Propositions \ref{u_t} and \ref{umaxH2}.
Hence, by density, we obtain the following relation for all $w \in H_0^2(\Omega)$:
\begin{equation}\label{statplate}
\langle\Delta \bar{u}, \Delta w\rangle - \left\langle \left[ \bar{u}, v(\bar{u}) + F_0 \right], w \right\rangle + U \left\langle \gamma [ \bar{\phi}], \partial_{x_1} w \right\rangle = \langle p_0, w\rangle.
\end{equation}
Similarly, multiplying the fluid part of equation \ref{flowplate} with $\psi\in C_0^{\infty}(\mathbb{R}^3_+)$ and integrating from $t_n$ to $t_n + a$, we get
\begin{align}\label{f}
\begin{split}
&\int_{t_n}^{t_n+a}(\phi_{tt}, \psi)dt+ \int_{t_n}^{t_n+a}U^2(\partial_{x_{1}}^2\phi,\psi)dt + \int_{t_n}^{t_n+a}(2U\partial_{x_{1}}\phi_t,\psi)dt=\int_{t_n}^{t_n+a}(\Delta\phi,\psi)dt\\
\end{split}
\end{align}
This implies
\begin{align}
\begin{split}
 & (\phi_t,\psi)\Big|_{t_n}^{t_n+a}-\int_{t_n}^{t_n+a}U^2(\partial_{x_1}\phi,\partial_{x_1}\psi)dt- 2U\int_{t_n}^{t_n+a}(\phi_t,\partial_{x_1} \psi)dt+\int_{t_n}^{t_n+a}(\nabla\phi,\nabla\psi)dt\\
&+\int_{t_n}^{t_n+a}\langle(\partial_t+U\partial_{x_1})u,\gamma[\psi]\rangle=0\\
\end{split}
\end{align}
We now estimate each term. Recall that $K_{\rho}\subset\subset \mathbb{R}^3_+$ contain the support of $\psi$. Then:
\begin{align}\label{estphi}
\begin{split}
U^2\int_{t_n}^{t_n+a}\Big |(\partial_{x_1}\phi-\partial_{x_1}\bar{\phi},\partial_{x_1}\psi)\Big | dt&\leq aU^2\|\partial_{x_1}\psi\| \max_{\tau\in [0,a]}\|\phi(t_n+\tau)-\bar{\phi}\|_{1,K_{\rho}}\\
2U\int_{t_n}^{t_n+a}\Big |(\phi_t,\partial_{x_1}\psi)\Big | dt&\leq 2aU\|\partial_{x_1}\psi\| \max_{\tau\in [0,a]}\|\phi_t(t_n+\tau)\|_{1,K_{\rho}}\\
\int_{t_n}^{t_n+a}\Big |(\nabla\phi-\nabla\bar{\phi},\nabla\psi)\Big | dt&\leq a\|\nabla\psi\| \max_{\tau\in [0,a]}\|\phi(t_n+\tau)-\bar{\phi}\|_{1,K_{\rho}}\\
\int_{t_n}^{t_n+a}\left|((\partial_t+U\partial_{x_1})u-U\partial_{x_1} \bar{u},\gamma[\psi])\right|&\leq a\max_{\tau>t_n}\|u_t(\tau)\|_1+a\max_{\tau\in[0,a]}\|u(t_n+\tau)-\bar{u}\|_2
\end{split}
\end{align}
Applying \eqref{phest1}--\eqref{phest4} to \eqref{estphi}, and again straightforwardly invoking Propositions~\ref{u_t} and \ref{umaxH2}, we see that each term on the right hand side of \eqref{estphi} approaches zero. Whence we obtain
\begin{equation}\label{statflow}
(\nabla \bar{\phi}, \nabla \psi) - U^2 (\partial_{x_1} \bar{\phi}, \partial_{x_1} \psi) + U\langle\partial_{x_1} \bar{u}, \gamma [\psi]\rangle = 0
\end{equation}
for any $\psi \in H^1(\mathbb R^3_+)$. Thus $(\bar{u}; \bar{\phi})$ satisfies \eqref{statplate} and \eqref{statflow} and is hence a weak solution.

\subsection{Step 4: Final Result}
We therefore have shown that any sequence $(u(t_n); \phi_t(t_n))$ with $t_n \to \infty$ contains a subsequence which converges to some stationary solution. 
We conclude by improving this convergence to the set $\mathcal N$.
\begin{proposition}\label{dumb} For $(u,\phi)$ a generalized solution to \eqref{flowplate} where $\phi_0,\phi_1$ have localized support in $K_{\rho}$, we have:
	\begin{equation}\label{refthisone}\lim_{t\rightarrow\infty}\inf_{\{\bar{u};\bar{\phi}\}\in\mathcal{N}}\Big\{||u(t)-\bar{u}(t)||_{2}^2+ ||u_t(t)||_{2}^2+||\phi(t)-\bar{\phi}(t)||_{H^1( K_{\tilde \rho} )}^2+||\phi_t(t)||_{L^2( K_{\tilde\rho} )}^2\Big\}=0,\end{equation}~~for any $\tilde \rho>0$.
\end{proposition} 

\begin{proof}[Proof of Proposition \ref{dumb}]
	\noindent
Assume the statement is not true. Then there is a sequence $t_n \to \infty$ and some $\epsilon>0$ so that  for all $n$ sufficiently large
\[
\inf_{\{\bar{u};\bar{\phi}\}\in\mathcal{N}}\Big\{||u(t_n)-\bar{u}(t_n)||_{2,\Omega}^2+ ||u_t(t_n)||_{2,\Omega}^2+||\phi(t_n)-\bar{\phi}(t_n)||_{H^1( K_{\rho} )}^2+||\phi_t(t_n)||_{L^2( K_{\rho} )}^2\Big\}>\epsilon.
\]
But for any such sequence $\{t_n\},$ we have shown that there exists a subsequence $\{t_{n_k}\}$ such that
\[
\lim_{k\rightarrow\infty}\inf_{\{\bar{u};\bar{\phi}\}\in\mathcal{N}}\Big\{||u(t_{n_k})-\bar{u}(t_{n_k})||_{2,\Omega}^2+ ||u_t(t_{n_k})||_{2,\Omega}^2+||\phi(t_{n_k})-\bar{\phi}(t_{n_k})||_{H^1( K_{\rho} )}^2+||\phi_t(t_{n_k})||_{L^2( K_{\rho} )}^2\Big\}\to 0,
\]
which is a contradiction. Hence 
\[
\lim_{t\rightarrow\infty}\inf_{\{\bar{u};\bar{\phi}\}\in\mathcal{N}}\Big\{||u(t)-\bar{u}(t)||_{2}^2+ ||u_t(t)||_{2}^2+||\phi(t)-\bar{\phi}(t)||_{H^1( K_{\rho} )}^2+||\phi_t(t)||_{L^2( K_{\rho} )}^2\Big\}=0
\]
\end{proof}
With the above claim, we conclude the proof of Theorem \ref{unibound}. In the case that $\mathscr N$ is isolated (e.g., finite), \eqref{refthisone} collapses to the result of Corollary \ref{improve}.

\section{Declarations}

\subsection{Funding} 
The second author would like to thank the National Science
Foundation, and acknowledge his partial funding from NSF
DMS-1907620 (Justin T. Webster).

\subsection{Conflicts of Interest/Competing Interests}
Not applicable.

\subsection{Availability of data and material} 
Not applicable.

\subsection{Code availability}
Not applicable.

\end{document}